\documentclass[a4paper, 12pt]{amsart}
\usepackage{amsfonts, amsthm, amssymb, amsmath}
\usepackage{mathrsfs,array}
\usepackage{xy}
\usepackage{hyperref}
\usepackage{verbatim}
\usepackage{cancel}
\usepackage[latin1]{inputenc}
\usepackage{tikz-cd}
\usepackage{bbm}
\usepackage{stmaryrd}

\input xy
\xyoption{all}

\newtheorem{theo}{Theorem}[section]
\newtheorem{prop}[theo]{Proposition}
\newtheorem{defi}[theo]{Definition}
\newtheorem{lemm}[theo]{Lemma}
\newtheorem{coro}[theo]{Corollary}

\newtheorem{exer}[theo]{Exercise}

\theoremstyle{definition}

\newtheorem{rema}[theo]{Remark}

\newcommand{\wh}{\widehat}
\newcommand{\wt}{\widetilde}
\newcommand{\mb}{\mathbb}
\newcommand{\mc}{\mathcal}
\newcommand{\mf}{\mathfrak}

\newcommand{\sub}{\subseteq}

\newcommand{\p}{\mathfrak{p}}

\newcommand{\Zp}{\mathbb{Z}_{p}}
\newcommand{\Qp}{\mathbb{Q}_{p}}

\newcommand{\Z}{\mathbb{Z}}

\newcommand{\ol}{\overline}

\newcommand{\bu}{\bullet}

\newcommand{\oo}{\mathcal{O}}

\newcommand{\A}{\mathbb{A}}

\newcommand{\spe}{\mathrm{sp}}
\newcommand{\m}{\mathfrak{m}}

\newcommand{\Ab}{\mathrm{Ab}}
\newcommand{\C}{\check{C}}
\newcommand{\cH}{\check{H}}
\newcommand{\lb}{[\![}
\newcommand{\rb}{]\!]}
\newcommand{\Gm}{\mathbb{G}_m}
\newcommand{\shom}{\mc{H}\mathrm{om}}

\DeclareMathOperator{\Ext}{Ext}
\DeclareMathOperator{\Tor}{Tor}

\DeclareMathOperator{\Hom}{Hom}

\DeclareMathOperator{\Coker}{Coker}

\DeclareMathOperator{\Spa}{Spa}
\DeclareMathOperator{\Spec}{Spec}
\DeclareMathOperator{\Spf}{Spf}

\DeclareMathOperator{\GL}{GL}

\DeclareMathOperator{\Ker}{Ker}

\DeclareMathOperator{\Sh}{Sh}
\DeclareMathOperator{\PSh}{PSh}
\DeclareMathOperator{\Ch}{Ch}
\DeclareMathOperator{\K}{K}
\DeclareMathOperator{\Aut}{Aut}

\begin{document}

\title{Topics in adic spaces}

\author{Christian Johansson}
\address{Christian Johansson\\Department of Mathematical Sciences, Chalmers University of Technology and the University of Gothenburg, 412 96 Gothenburg, Sweden}
\email{chrjohv@chalmers.se}
\urladdr{http://www.math.chalmers.se/~chrjohv/}

\begin{abstract}
These are expanded notes from a four lecture mini-course given by the author at the Spring School on Non-archimedean geometry and Eigenvarieties, held at the University of Heidelberg in March 2023. The course discusses coherent sheaves, various properties of morphisms as well as the links to formal geometry and Berkovich spaces.
\end{abstract}

\maketitle

\section*{Introduction}

These are expanded lecture notes from a four lecture mini-course during the Spring School on Non-archimedean geometry and Eigenvarieties, held at the University of Heidelberg in March 2023. The course was the last of three mini-courses, following \cite{bergdall-heid,hubner-heid}, that introduced the theory of adic spaces and prepared for courses on eigenvarieties during the second week of the spring school. The notes have been revised and expanded after the Spring School, but is still based around the material covered in the four lectures given. The material a slightly disjoint collection of topics:

\begin{itemize}
\item Coherent sheaves;

\item Properties of morphisms;

\item Quasi-Stein spaces and the link to formal geometry;

\item Uniformization of curves and Berkovich spaces.
\end{itemize}  

Roughly speaking, these four items correspond to the four lectures. Broadly speaking, coherent sheaves, morphisms and quasi-Stein spaces are both important on their own and serve as preparation for the theory of eigenvarieties, while the remaining topics serve to illustrate the links between the theory of adic spaces and other theories of nonarchimedean geometry, focusing on curves to get a concrete illustration.

What we cover is neither comprehensive nor original, and we do not give full proofs. Instead, we have tried to mix proofs with giving an overview and some intuition. There are many excellent sources that we recommend, including \cite{hu-etale,bosch,fresnel-vanderput,berkeleylectures,aws-perfectoid}. In particular, the book \cite{fresnel-vanderput} contains a wealth of interesting geometric examples and applications.

Exercises were sprinkled throughout the text, for the benefit of the Spring School participants. For the revision, we have decided to mostly keep them in their original form. Some exercises aim to fill in details in the text and others focus on trying to develop intuition. For some exercises, we do not necessarily expect that the reader will be able to fill in all details without perhaps consulting the literature, and we encourage the reader to do so as needed.

\subsection*{Notation and conventions}

We assume the basic concepts in the theory of adic spaces, as covered in the lectures of Bergdall and H\"ubner \cite{bergdall-heid,hubner-heid}. To simplify notation, we will write $\Spa A$ for $\Spa(A,A^\circ)$, and more often than not we will consider affinoid adic spaces with $A^+ = A^\circ$. Moreover, we will assume that all Huber rings $A$ that we encounter are complete unless otherwise stated.

Roughly following \cite{hu-etale}, we will call a Huber ring Noetherian if it is either 
\begin{enumerate}
\item Finitely generated (as a ring) over a Noetherian ring of definition, or 

\item A strongly Noetherian Tate ring.
\end{enumerate}

When writing ``strongly Noetherian'', we will always implicitly assume that the Huber ring is Tate. We say that an adic space $X$ is locally Noetherian if it has an affinoid open cover $U_i=\Spa(A_i,A_i^+)$, where the $A_i$ are either all of type (1) above, or all of type (2) above. Throughout the text, a rigid space over a nonarchimedean field $K$ will be an adic space which is locally of the form $\Spa A$, with $A$ topologically of finite type over $K$. In general, when discussing an adic space $X$ of a certain type, all affinoid open subsets $U\sub X$ that we consider will be of the same type as $X$ (for example, when considering a rigid space $X$, an affinoid open subset $U\sub X$ will always mean an affinoid rigid space).  

Recall that a Huber pair $(A,A^+)$ is called \emph{sheafy} if the structure presheaf $\oo_X$ on the topological space $X=\Spa(A,A^+)$ is a sheaf. We call a sheafy $(A,A^+)$ \emph{acyclic} if the higher \v{C}ech cohomology of $\oo_X$ vanishes for all finite covers of $X$ by rational subsets\footnote{As far as the author is aware, all known sheafy $(A,A^+)$ are acyclic, and sheafiness and acyclicity are usually proved simultaneously. If $A$ is analytic (meaning that $A$ itself is the only open ideal in $A$; this includes Tate rings), then it is known that sheafiness implies acyclicity, and sheafiness of $(A,A^+)$ only depends on $A$; see \cite[Theorem 1.3.4, Remark 1.6.9]{kedlaya-aws}.}.

Given an adic space $X$, we will sometimes denote its underlying topological space by $|X|$, and sometimes we will simply write $X$ when we hope that no confusion can arise (the observant reader will note that we have already used $\Spa(A,A^+)$ to denote an adic space as well as the underlying topological space).

Two final pieces of notation: If $X$ is a topological space and $S \sub X$ is a subset, the closure of $S$ is denoted by $\ol{S}$, and if $K$ is a nonarchimedean field, then we will write $\pi$ for a pseudouniformizer (i.e. a topologically nilpotent unit) and $k$ for the residue field.

\subsection*{Acknowledgments}
I wish to thank Eugen Hellmann, Judith Ludwig, Sujatha Ramdorai, and Otmar Venjakob for organizing the Spring School and for the invitation to give this lecture series. I also wish to thank John Bergdall, Eugen Hellmann, Ben Heuer, Katharina H\"ubner, Judith Ludwig and James Newton for stimulating discussion and feedback before and during the Spring School, and Rustam Steingart for his work as a course assistant.  Thanks are due to the the referee and to Jacksyn Bakeberg for helpful comments and pointing out errors in a previous version of these notes. A special thanks goes to Alex Youcis for numerous insightful comments, which greatly improved the notes.

Last, but certainly not least, I wish to thank all participants of the Spring School for their questions, comments, corrections and energy during the Spring School.

The author was supported by Vetenskapsr\r{a}det Grant 2020-05016, \emph{Geometric structures in the $p$-adic Langlands program}, during the preparation of these notes.

\newpage

\section{Lecture 1: Coherent sheaves}

\subsection{Brief recollections on sheaf theory}

The goal of this lecture is to define coherent sheaves on a locally Noetherian adic space, and prove that the definition is robust. For sheaf theory on an adic space $X$, we always work with abelian sheaves on the underlying topological space. A sheaf is determined by its values on any basis of the topology, and it is often technically convenient to restrict one's attention to a suitable basis; we will do so without commenting further. In particular, on an affinoid adic space, we will describe sheaves by their values on rational subsets.

Just as when discussing sheafiness of the structure presheaf, it is convenient to include cohomology straight away in the discussion of coherent sheaves. For adic spaces, \v{C}ech cohomology is a basic tool in the beginning of the theory, and derived functor cohomology comes into the picture once one has established enough information about \v{C}ech cohomology. We assume that the reader is familiar with basic notions of \v{C}ech cohomology (if not, many sources define the basics, for example \cite[Chapter III]{hartshorne}). Given a presheaf $\mc{F}$ on a topological space $X$ and an open cover $\mf{U} = (U_i)_{i\in I}$ of $X$, we let 
\[
\C^\bu(\mf{U},\mc{F}) \,\,\,\, \text{and} \,\,\,\, \C_{aug}^\bu(\mf{U},\mc{F})
\]
denote the \v{C}ech complex and the augmented \v{C}ech complex of $\mc{F}$ with respect to $\mf{U}$, respectively. Our convention is that $\C^\bu(\mf{U},\mc{F})$ has terms in (cohomological) degrees $\geq 0$, and that $\C_{aug}^\bu(\mf{U},\mc{F})$ additionally has the term $\mc{F}(X)$ in degree $-1$. We denote the \v{C}ech cohomology of $\mc{F}$ with respect to $\mf{U}$ by $\cH^i(\mf{U},\mc{F})$. Taking the direct limit over all coverings $\mf{U}$ of $X$, we obtain the (cover-independent) \v{C}ech cohomology $\cH^i(X,\mc{F})$ of $\mc{F}$.

For our needs, the relationship between \v{C}ech cohomology and sheaf cohomology is summarized in the following two results. We sketch the proof of these results from first principles (using basic homological algebra) in Appendix \ref{appendix on coh}; see Corollaries \ref{Cartan} and \ref{Leray}.

\begin{theo}[Cartan]\label{Cartan main}
Let $\mc{F}$ be a sheaf on $X$, and $\mf{B}$ be a basis of open sets of $X$ which is closed under intersection. Assume that $\cH^i(U,\mc{F}) =0$ for any open $U \in \mf{B}$ and any $i\geq 1$. Then the natural map $\cH^i(\mf{U},\mc{F}) \to H^i(X,\mc{F})$ is an isomorphism for any open cover $\mf{U} \sub \mf{B}$ and any $i \geq 0$. 
\end{theo}

\begin{theo}[Leray's Theorem]\label{Leray main}
Let $\mf{U}=(U_i)_{i\in I}$ be an open cover of $X$ and let $\mc{F}$ be a sheaf on $X$. Assume that $H^j(U,\mc{F})=0$ for any finite intersection $U$ of sets in $\mf{U}$ and any $j\geq 1$. Then the natural map $\cH^j(\mf{U},\mc{F}) \to H^j(X,\mc{F})$ is an isomorphism for all $j\geq 0$.
\end{theo}

\subsection{Sheaves associated with modules} We start our discussion with a sheafy and acyclic Huber pair $(A,A^+)$. In all cases we consider, we recall that what has been proved is that the augmented \v{C}ech complex $\C^\bu_{aug}(\mf{U},\oo_X)$ is acyclic for the structure presheaf $\oo_X$ and all finite covers of $X=\Spa(A,A^+)$ by rational subsets. This both implies that $\oo_X$ is a sheaf and, by Theorem \ref{Cartan main}, that $H^i(X,\oo_X)=0$ for all $i\geq 1$. A key ingredient in setting up the theory of quasicoherent sheaves in algebraic geometry is flatness of localization. By contrast, if $U\sub X$ is a rational subset (with $X$ as above), the restriction map $A \to \oo_X(U)$ need not be flat. This is a serious problem in trying to set up a theory of quasicoherent sheaves for adic spaces, though not the only one. We recall the following positive result:

\begin{theo}
Suppose that $(A,A^+)$ is a Noetherian Huber pair and that $U\sub X=\Spa(A,A^+)$ is a rational subset. Then $A \to \oo_X(U)$ is flat.
\end{theo}

A reference\footnote{The reader may take the case when $A$ has a Noetherian ring of definition as an exercise, using flatness of $I$-adic completion for Noetherian rings \cite[Tag 00MB]{stacks-project}. \emph{Hint: If $A$ is a (not necessarily complete) Huber ring with ring of definition $A_0$, then the completion $\wh{A}$ is equal to $\wh{A}_0 \otimes_{A_0} A$.}} is \cite[Lemma 1.7.6]{hu-etale}; see also \cite[Theorem 1.4.14]{kedlaya-aws}. In particular, we should have a better chance of succeeding when $X$ is locally Noetherian. A further advantage in this situation is the following remark on canonical topologies (see e.g \cite[Lemma 2.3]{hu-generalization} and the discussion preceding it).

\begin{prop}
Assume that $(A,A^+)$ is a Noetherian Huber pair and that $M$ is a finitely generated $A$-module. If $A^n \to M$ is a surjection, we equip $M$ with the quotient topology coming from this surjection, where $A^n$ has the product topology. This topology is called the canonical topology and is complete and independent of the choice of surjection from a finitely generated free module. Moreover, any $A$-module map $M \to N$ between finitely generated $A$-modules is continuous and strict\footnote{Recall $f : M \to N$ is strict if the subspace topology on $\mathrm{Im}(f)$ coming from the inclusion into $N$ coindices with the quotient topology on $\mathrm{Im}(f)$ coming from the surjection from $M$.} for the canonical topology.
\end{prop} 

We start our discussion of quasicoherent sheaves by imitating the construction in algebraic geometry.

\begin{prop}\label{quasicoherent sheaves on affines}
Let $(A,A^+)$ be a Noetherian Huber pair and let $M$ be an $A$-module. Then the presheaf $\wt{M}(U) = M \otimes_A \oo_X(U)$ on $X=\Spa(A,A^+)$ is a sheaf, and $H^i(X,\wt{M})=0$ for all $i\geq 1$. Moreover, the assignment $M \mapsto \wt{M}$ defines an exact functor from $A$-modules to $\oo_X$-modules.
\end{prop}

\begin{proof}
Let $\mf{U}$ be any finite cover of $X$ by rational subsets; we need to prove that $\C_{aug}^\bu(\mf{U},\wt{M})$ is acyclic. First, note that if $M=A^{\oplus I}$ is a free module (for some index set $I$), then $\C_{aug}^\bu(\mf{U},\wt{M}) = \C_{aug}^\bu(\mf{U},\oo_X)^{\oplus I}$, which is acylic since $\C_{aug}^\bu(\mf{U},\oo_X)$ is acyclic and direct sums are exact.

For a fixed $\mf{U}$ we prove that $\cH^i_{aug}(\mf{U},\wt{M}) = 0$ for all $M$ and all $i\geq -1$ by descending induction on $i$. For large enough $i$ this is clear since $\mf{U}$ is a finite cover. Assume that $\cH^i_{aug}(\mf{U},\wt{M}) = 0$ for all $M$ and all $i \geq d$, for some $d\geq 0$. Let $N$ be any $A$-module and choose a surjection $F \to N$ from a free module $F$; let $K$ be the kernel. Since all maps $A \to \oo_X(U)$ are flat when $(A,A^+)$ is Noetherian, the exact sequence $0 \to K \to F \to N \to 0 $ gives an exact sequence
\[
0 \to \C^\bu_{aug}(\mf{U},\wt{K}) \to \C^\bu_{aug}(\mf{U},\wt{F}) \to \C^\bu_{aug}(\mf{U},\wt{N}) \to 0.
\]
We know that $\C^\bu_{aug}(\mf{U},\wt{F})$ is acylic since $F$ is free, so the long exact sequence for cohomology and the induction hypothesis gives us that $\cH^{d-1}_{aug}(\mf{U},\wt{N}) = \cH^d_{aug}(\mf{U},\wt{K}) = 0$. This finishes the proof that the $\wt{M}$ are acyclic sheaves. For the last part, note that $M \mapsto \wt{M}$ is even exact as a functor into presheaves (on rational subsets).
\end{proof}

\begin{exer}
For general sheafy and acyclic $(A,A^+)$, prove that $\wt{M}$ is an acyclic sheaf when $M$ is a \emph{flat} $A$-module.
\end{exer}

In light of Proposition \ref{quasicoherent sheaves on affines}, it is tempting to define a sheaf $\mc{F}$ on a locally Noetherian adic space $X$ to be quasicoherent if there is a cover $(U_i = \Spa(A_i,A_i^+))_{i\in I}$ of affinoid opens such that for all $i$ there is an $A_i$-module $M_i$ such that $\mc{F}|_{U_i} \cong \wt{M_i}$. However, this definition has some problems:

\begin{enumerate}
\item Does it depend on the choice of the cover? This boils down to the following question: If $X=\Spa(A,A^+)$ is affinoid Noetherian itself and $\mc{F}$ is quasicoherent in the sense above, is $\mc{F} \cong \wt{M}$ for some $A$-module $M$?

\item Let $U\sub X$ be a rational open subset. Since $\oo_X(U)$ is a \emph{complete topological ring} (and we have actively used completions in the definition of $\oo_X(U)$), one might reasonably demand that a quasicoherent sheaf should take values in complete topological $\oo_X(U)$-modules on $U$. However, the sheaves $\wt{M}$ above have no topology a priori, and even if we demand that $M$ is a complete topological $A$-module, $M \otimes_A \oo_X(U)$ need not be complete. One can try to get around this by only looking at complete topological $A$-modules $M$ and redefining $\wt{M}$ using a completed tensor product $M \wh{\otimes}_A \oo_X(U)$, but this runs into the problem that $-\wh{\otimes}_A \oo_X(U)$ need not be exact even if $-\otimes_A \oo_X(U)$ is.
\end{enumerate}

There have been various approaches to getting around this problem (as well as the problem that Proposition \ref{quasicoherent sheaves on affines} does not hold for many modules of interest if rational localizations fail to be flat). These include ad hoc notions of ``Banach sheaves'' on rigid analytic varieties \cite{aip-siegel} and the notion of pseudocoherent sheaves of Kedlaya--Liu \cite{kedlaya-liu-ii,kedlaya-aws}; see in particular \cite[Theorem 1.4.18]{kedlaya-aws}. More recently, Clausen and Scholze's theory of condensed mathematics has given a solution both to the problem of sheafiness in the theory of adic spaces, and the problem with quasicoherent sheaves, within their larger framework of ``analytic geometry''. The interested reader might look at \cite{scholze-condensed,scholze-analytic,clause-scholze-lectures} for an introduction to this remarkable field, and at \cite{andreychev,mikami,scholze-mathoverflow} for more details and results on sheafiness and quasicoherent sheaves on (analytic) adic spaces given by analytic geometry.

\subsection{Coherent sheaves}

The traditional way to solve the problems above is focus on \emph{finitely generated} modules, i.e. \emph{coherent sheaves}. We go back to the setting of Proposition \ref{quasicoherent sheaves on affines}. If $M$ is finitely generated over $A$, then $M \otimes_A \oo_X(U)$ is finitely generated over $\oo_X(U)$ for every $U$ and hence carries the canonical topology, making $\wt{M}$ a sheaf of complete topological $\oo_X$-modules. This solves the second problem. The first can (importantly!) also be solved, though it will take more work.

\begin{theo}\label{coherence is local}
Let $(A,A^+)$ be a Noetherian Huber pair with $X=\Spa(A,A^+)$ and let $\mc{F}$ be a sheaf of $\oo_X$-modules on $X$. Assume that there is a cover $(U_i)_{i\in I}$ of rational subsets and finitely generated $\oo_X(U_i)$-modules $M_i$ such that $\mc{F}|_{U_i} \cong \wt{M_i}$ for all $i$. Then there is a finitely generated $A$-module $M$ such that $\mc{F} \cong \wt{M}$. 
\end{theo}

\begin{exer}
In the setting of Theorem \ref{coherence is local} (and admitting the result), prove that $M$ is projective if and only if $M_i$ is projective for all $i$. Coherent sheaves satisfying this condition are called vector bundles, and they behave well on general (not necessarily locally Noetherian) analytic adic spaces; see \cite[Theorem 1.4.2]{kedlaya-aws}.
\end{exer}

The proof of Theorem \ref{coherence is local} is somewhat elaborate. We will give it in the case when $A$ is strongly Noetherian, \textbf{which we now assume until Definition \ref{definition of coherence}}. Set
\[
M := \mc{F}(X).
\] 
If the theorem holds, then clearly we must have $M=\mc{F}(X)$, so this is forced on us. However, it is not so clear that $M$ is finitely generated, nor that the natural maps $M \otimes_A \oo_X(U) \to \mc{F}(U)$ are isomorphisms for rational subsets $U$ of $X$. We will prove this in steps.

As in the proof of sheafiness, the two broad steps of the proof are to first reduce to the case of a simple Laurent cover, and then compute directly there. The reduction is exactly the same as in the proof of sheafiness: We may refine the cover to a Laurent cover and then, by induction, we may reduce to a simple Laurent cover.

So, we assume that $I=\{1,2\}$ and that there is function $f\in A$ such that $U_1 = \{ |f|\leq 1\}$, $U_2 = \{ |f| \geq 1\}$. We set $U_{12} = U_1 \cap U_2$, $M_{12} = \mc{F}(U_{12})$, $A_i = \oo_X(U_i)$ and $A_{12} = \oo_X(U_{12})$. We have restriction maps
\[
\psi_i : M_i \to M_{12}
\]
which induce isomorphisms $M_{12} \cong M_i \otimes_{A_i} A_{12}$, and $M$ is the kernel of the map $\psi_1 - \psi_2 : M_1 \times M_2 \to M_{12}$. A key feature of a simple Laurent cover is that the map $A_2 \to A_{12}$ has dense image:

\begin{lemm}\label{cartans lemma} The following hold:
\begin{enumerate}
\item The map $A_2 \to A_{12}$ has dense image.

\item There is a constant $c>0$ such that for every positive integer $n$, every matrix $B \in \GL_n(A_{12})$ with $|B-1|\leq c$ can be written in form $B_1 B_2$ for matrices $B_i \in \GL_n(A_i)$.  
\end{enumerate}

\end{lemm}

\begin{proof}
For part (1), note that $A_2 = A\langle f^{-1} \rangle$ contains $A[f^{-1}]=A[f,f^{-1}]$, which is dense in $A_{12} = A\langle f,f^{-1} \rangle$. 

It remains to prove part (2). Equip $A_1$, $A_2$ and $A_{12}$ with submultiplicative norms that define the topology; matrices are then given the maximum norm with respect to the entries (note that this norm is submultiplicative as well). Let $r_i : A_i \to A_{12}$ be the restriction maps and let $r = r_1 -r_2 : A_1 \times A_2 \to A_{12}$ be the difference map; we further assume that the norms have been chosen so that $|r(x,y)|\leq \max(|x|,|y|)$. 

The map $r$ is surjective by acyclicity, so by the open mapping theorem there is a constant $d \in (0,1)$ such that any $x \in A_{12}$ can be written as $x=r(y_1,y_2)$ for some $(y_1,y_2) \in A_1 \times A_2$ with $d|y_i| \leq |x|$. We set $c= d^3$. Write $B=I+V_1$, then $|V_1|\leq c$. We can write $V_1 = r(C_1,D_1)$ with $|C_1|,|D_1|\leq d^{-1}|V_1|$. Now define matrices $V_n$, $C_n$ and $C'D_n$ recursively for $n\geq 2$ by
\[
1+V_{n} = (1-r_1(C_{n-1}))(1+V_{n-1})(1+r_2(D_{n-1}))
\]
and $r(C_n,D_n) = V_n$ with $|C_n|,|D_n|\leq d^{-1}|V_n|$ (these conditions do not define $(B_n,C_n)$ uniquely, but any choice will do). By induction one shows that $|V_n| \leq d^{n+2}$ and $|C_n|,|D_n| \leq d^{n+1}$, so the infinite products
\[
B_1 = \prod_{n=1}^\infty (1-C_n)\,\,\,\,\, \text{and}\,\,\,\,\, B_2 =\prod_{n=1}^{\infty} (1+D_n)
\]
converge to invertible matrices and satisfy $r_1(B_1)Br_2(B_2) = I$. It follows that $B = r_1(B_1^{-1})r_2(B_2^{-1})$, as desired.
\end{proof}

To prove Theorem \ref{coherence is local}, we want to show the following:

\begin{itemize}
\item $M$ is finitely generated over $A$.

\item The natural maps $M \otimes_A A_i \to M_i$ are isomorphisms.
\end{itemize} 

We start with the following lemma, cf. \cite[Lemma 2.7.4]{kedlaya-liu-i}.

\begin{lemm}\label{coherence lemma 1}
The maps $M \otimes_A A_i \to M_i$ are surjective, and $\psi_1- \psi_2$ is surjective.
\end{lemm}

\begin{proof}
Choose generating sets $u_1,\dots,u_n$ of $M_1$ and $v_1,\dots,v_n$ of $M_2$. Then we may choose matrices $B,C \in M_n(R_{12})$ such that $\psi_1(u_j) = \sum_{i=1}^n C_{ij}\psi_2(v_i)$ and $\psi_2(v_j) = \sum_{i=1}^n B_{ij}\psi_1(u_i)$. Since $A_2 \to A_{12}$ has dense image (by Lemma \ref{cartans lemma}), we may choose a matrix $C^\prime \in M_n(A_2)$ such that $B(C^\prime -C)$ has norm $\leq c$, where $c$ is the constant in Lemma \ref{cartans lemma}(2). We may then write $I+B(C^\prime-C) = D E^{-1}$ with $D \in \GL_n(A_1)$, $E \in \GL_n(A_2)$, and we define elements $x_j = (y_j,z_j) \in M_1\times M_2$ for $j=1,\dots,n$ by
\[
x_j = (y_j,z_j) = \left( \sum_{i=1}^n D_{ij}u_i, \sum_{i=1}^n (C^\prime E)_{ij} v_i \right).
\]
Then 
\[
\psi_1(x_j) - \psi_2(x_j) = \sum_{i=1}^n (D - BC^\prime E)_{ij}\psi_1(u_i) = \sum_{i=1}^n ((1-BC)E)_{ij}\psi_1(u_i) = 0,
\]
so $x_j \in M$. Since $D \in \GL_n(A_1)$, the elements $y_j$ generate $M_1$ over $A_1$, and it follows that $M \otimes_A A_1 \to M_1$ is surjective. Applying $-\otimes_{A_1}A_{12}$ to this surjection, we get a surjection $M\otimes_A A_{12} \to M_{12}$. Since $A_1 \times A_2 \to A_{12}$ is surjective, we deduce that we have a surjection $M\otimes_A (A_1 \times A_2) \to M_{12}$. Since this map factors through $\psi_1 -\psi_2$, we deduce that $\psi_1 -\psi_2$ is surjective as well.

It remains to prove that $M \otimes_A A_2 \to M_2$ is surjective. Let $m \in M_2$. By above, $\psi_2(m)$ lifts to $M\otimes_A (A_1 \times A_2)$, i.e. we can find $m_i$ in the image of $M \otimes_A A_i \to M_i$ such that $\psi_2(m) = \psi_1(m_1) -\psi_2(m_2)$. Put $m^\prime = (m_1,m+m_2)\in M_1 \times M_2$ and note that $m^\prime \in M$ by construction. In particular, the image of $M\otimes_A A_2 \to M_2$ contains $m_2$ and the image of $m^\prime \otimes 1$, i.e. $m+m_2$, and hence also $m$ as desired. This finishes the proof.
\end{proof}

We now finish the proof of Theorem \ref{coherence is local}.

\begin{lemm}
$M$ is a finitely generated $A$-module, and $\mc{F} \cong \wt{M}$.
\end{lemm}

\begin{proof}
Choose a finitely generated free $A$-module $F$ and a map $F \to M$ such that the compositions $F_i := F \otimes_A A_i \to M \otimes_A A_i \to M_i$ are surjective (this is possible by Lemma \ref{coherence lemma 1}). Setting $F_{12} = F \otimes_A A_{12}$, the induced map $F_{12} \to M_{12}$ is surjective as well. Let $N_i$ be the kernel of $F_i \to M_i$, let $N_{12}$ be the kernel of $F_{12} \to M_{12}$, and let $N$ be the kernel of $N_1 \times N_2 \to N_{12}$. Consider the diagram
\[
\xymatrix{
0 \ar[r] & F \ar[r] \ar[d] & F_1 \times F_2 \ar[r] \ar[d] & F_{12} \ar[r] \ar[d] & 0\\
0 \ar[r] & M \ar[r] & M_1 \times M_2 \ar[r] & M_{12} \ar[r] & 0 
}
\]
with exact rows. Let $K = \Ker(F \to M)$ and $C = \Coker(F \to M)$. Applying the snake lemma we get an exact sequence
\[
0 \to K \to N_1 \times N_2 \to N_{12} \to C \to 0,
\]
from which we deduce that $K=N$. Moreover, note that the map $F \to M$ induces a map $\wt{F} \to \mc{F}$ and that the kernel $\mc{N}$ of this latter map is a sheaf satisfying $\mc{N}|_{U_i} = \wt{N_i}$ and $\mc{N}(U_{12})=N_{12}$. In particular, it is of the same type as $\mc{F}$, so we may apply Lemma \ref{coherence lemma 1} to it to deduce that $N_1 \times N_2 \to N_{12}$ is surjective. It then follows that $C=0$, so $M$ is finitely generated as desired. To finish, note that $\wt{F} \to \mc{F}$ is surjective, since it is surjective on $U_1$ and $U_2$. Since $\mc{N}$ is a sheaf of the same type as $\mc{F}$, we can choose a finite free $A$-module $G$ and a surjection $\wt{G} \to \mc{N}$. But then we have a presentation
\[
\wt{G} \to \wt{F} \to \mc{F} \to 0,
\]
so $\mc{F}$ is the cokernel of $\wt{G} \to \wt{F}$. But the cokernel of this map is just the sheaf associated with the $A$-module $\Coker(G \to F)$, so $\mc{F}$ is the sheaf associated with an $A$-module, and that $A$-module must be $M$. 
\end{proof}

Armed with Theorem \ref{coherence is local}, we can now define coherent sheaves on locally Noetherian adic spaces.

\begin{defi}\label{definition of coherence}
Let $X$ be a locally Noetherian adic space. An $\oo_X$-module $\mc{F}$ is called coherent if there is a cover $(U_i)_{i\in I}$ of $X$ by affinoid opens $U_i = \Spa(A_i,A^+_i)$ and finitely generated $A_i$-modules $M_i$ such that $\mc{F}|_{U_i} \cong \wt{M_i}$ for all $i$.
\end{defi}

\begin{exer}
Check that, if $U=\Spa(A,A^+) \sub X$ is any open Noetherian affinoid adic space, then $\mc{F}|_U \cong \wt{M}$ for some finitely generated $A$-module $M$.
\end{exer}

In our discussion so far, we have only explicitly defined the value of our coherent sheaves on affinoids and their rational subsets. We finish by spelling out how to do define their values on arbitrary opens, as \emph{topological} modules.

So, suppose that we have a coherent sheaf $\mc{F}$ on a locally Noetherian adic space $X$. For every open affinoid $U=\Spa(A,A^+) \sub X$ with $A$ Noetherian, $\mc{F}(U)$ is therefore a finitely generated $A$-module and hence carries the canonical topology. The collection of such $U$ form a basis of neighbourhoods of $X$, and for a general open $V\sub X$ we define 
\[
\mc{F}(V) := \lim_{U\sub V} \mc{F}(U)
\]
as the \emph{limit} of the $\mc{F}(U)$ as $U$ ranges over the affinoid Noetherian opens contained in $V$, in the category of \emph{topological} $\oo_X(V)$-modules\footnote{This presupposes that we have topologized $\oo_X(V)$; this is done by taking the same limit for $\mc{F}=\oo_X$ in their category of topological rings.} (note that this need not be an inverse limit). With this definition, the restriction maps become continuous maps.

Let us spell out the result of this more carefully when $X$ is a rigid space over $K$. In this case, we only consider $U\sub X$ that are affinoid rigid spaces, and then the $\mc{F}(U)$ are $K$-Banach spaces. In particular, if $V\sub X$ is a quasicompact open with open cover $V = U_1 \cup \dots \cup U_r$ by affinoid rigid spaces, then $\mc{F}(V)$ is a closed subspace of $\prod_{i=1}^r \mc{F}(U_i)$ and hence a Banach space. For a general open $V\sub X$, we can write
\[
\mc{F}(V) = \varprojlim_{W\sub V} \mc{F}(W)
\]
where the \emph{inverse} limit runs over all quasicompact open $W\sub V$. In particular, if $V$ is a countable union of quasicompact opens (which is almost always the case in practice), $\mc{F}(V)$ is a Fr\'echet space.  

\newpage

\section{Lecture 2: Finiteness properties of morphisms}

In the next lecture and a bit, we will define many basic properties of morphisms of adic spaces, analogous to the basic ones used in algebraic geometry. In this lecture we focus on finiteness properties, starting with finite morphisms. The definitions we give are not always the standard definitions given; for simplicity we often give equivalent characterizations that are (sometimes) more convenient to use.

\begin{defi}
Let $f : X \to Y$ be a morphism of locally Noetherian adic spaces.

\begin{enumerate}
\item We say that $f$ is finite if there is a cover $(U_i)_{i\in I}$ of $Y$ by affinoid opens $U_i = \Spa(A_i,A^+_i)$ such that $V_i = f^{-1}(U_i)$ is an affinoid adic space $\Spa(B_i,B_i^+)$ with $B_i$ finitely generated as an $A_i$-module (with its canonical topology) and $B_i^+$ is the integral closure of $A_i^+$ inside $B_i$.

\item We say that $f$ is a Zariski closed immersion if $f$ is finite and, taking a cover as in (1) and using the same notation, the maps $A_i \to B_i$ are surjective.
\end{enumerate}
\end{defi}

\begin{exer}
Let $X$ be a locally Noetherian adic space. Show that finite morphisms $Y \to X$ correspond to coherent $\oo_X$-algebras, and that Zariski closed immersions correspond to those coherent $\oo_X$-algebras which are quotients of $\oo_X$, or, equivalently, to coherent ideals $\mc{I} \sub \oo_X$.
\end{exer}

\begin{rema}
One can define finite morphisms in larger generality. Moreover, Zariski closed immersions can be used to define a ``Zariski topology'' on a locally Noetherian adic space, taking as closed subsets the (images of) Zariski closed immersions. This can sometimes be very useful, though one has to be a bit careful when using it as its interaction with the usual topology is slightly subtle. 
\end{rema}

Just as in algebraic geometry, finite morphisms sit inside the larger class of proper morphisms. In algebraic geometry, a key feature of proper morphisms is that they satisfy the so-called valuative criterion of properness. We would like to start by considering the analogue of this in adic geometry, but first we need to define morphisms (locally) of finite type. Huber defines a plethora of finiteness conditions in \cite{hu-etale}, and we will not be able to cover them all. Instead, we will make the following simplifying assumption on our adic spaces.

\textbf{Assumption: For the rest of this lecture, all adic spaces will locally be of the form $X=\Spa(A,A^+)$ for $A$ strongly Noetherian.} 

With this assumption, we make the following definitions. Recall that the notion of a morphism $(A,A^+) \to (B,B^+)$ being of topologically finite type has been defined in H\"ubner's lectures \cite[Definition 7.10]{hubner-heid}.

\begin{defi} Let $X$ and $Y$ be adic spaces.
\begin{enumerate}

\item A morphism $f : X \to Y$ is locally of finite type if for every point $x\in X$ there is an open Noetherian affinoid subset $V = \Spa(A,A^+) \sub Y$ and an open affinoid subset $U = \Spa(B,B^+)\sub X$ containing $x$ such that $f(U) \sub V$ and the induced map $(A,A^+) \to (B,B^+)$ is of topologically finite type.

\item A morphism $f : X \to Y$ is of finite type if it is locally of finite type and quasicompact.  
\end{enumerate}
\end{defi}

In particular, if $K$ is a nonarchimedean field, then a rigid space over $K$ is precisely an adic space $X \to \Spa(K)$ which is locally of finite type. In particular, all morphisms between rigid spaces over $K$ are locally of finite type. Next, we wish to discuss separated, proper and partially proper morphisms. These have a number of equivalent definitions, just like in algebraic geometry. We give one in terms of a valuative criterion.

\begin{defi}\label{sep and prop: adic} Let $f : X \to Y$ be a quasiseparated morphism of adic spaces which is locally of finite type, and consider a commuting diagram
\[
\xymatrix{
\Spa(K,K^\circ) \ar[r]\ar[d] & X  \ar[d]^f \\
\Spa(K,K^+) \ar[r] & Y
}
\]
where $K$ is a nonarchimedean field and $K^+\sub K^\circ$ is an open \emph{valutation} subring.
\begin{enumerate}
\item We say that $f$ is separated if, for any diagram as above, there is at most one morphism $\Spa(K,K^+) \to X$ making the diagram commute.

\item We say that $f$ is partially proper if, for any diagram as above, there is exactly one morphism $\Spa(K,K^+) \to X$ making the diagram commute.

\item We say that $f$ is proper if it is partially proper and quasicompact.
\end{enumerate}
\end{defi}

\begin{rema}
Huber defines separatedness, partial properness and properness under a weaker initial finiteness assumption on $f$ called locally $^+$-weakly finite type. We refer to \cite{hu-etale} for details. We also remark that if $f$ is partially proper, then in fact, for any commutative diagram
\[
\xymatrix{
\Spa(R,R^\circ) \ar[r]\ar[d] & X  \ar[d]^f \\
\Spa(R,R^+) \ar[r] & Y
}
\]
with $R$ strongly Noetherian, there exists a unique morphism $\Spa(R,R^+) \to X$ making the diagram commute (see e.g. \cite[Proposition 6.13]{ludwig-notes}).
\end{rema}

Note that Definition \ref{sep and prop: adic} crucially uses higher rank points, and does not have an obvious analogue in the world of classical rigid geometry. We will now recall Kiehl's original definitions for rigid spaces (still phrased in the language of adic spaces), starting with separatedness. Let $K$ be a nonarchimedean field.

\begin{defi}\label{sep: rigid}
A morphism $X \to Y$ of rigid spaces over $K$ is called separated if the diagonal morphism $X \to X \times_Y X$ is a Zariski closed immersion. We say that $X$ is separated if the structure map $X \to \Spa(K)$ is separated.
\end{defi}

This is the definition that should be familiar from the world of schemes, and it has some familiar consequences (for example, the intersection of two affinoid opens inside a separated rigid space is affinoid). It turns out to be equivalent to the definition given in Definition \ref{sep and prop: adic}(1), though this requires some work \cite[Remark 1.3.19(ii)]{hu-etale}. We now move on to properness and partially properness. The ``rigid'' definitions rely on the notion of relative compactness of open subsets.

\begin{defi}
Let $X \to Y$ be a map of rigid spaces over $K$. Assume that $Y$ is affinoid and that $U \sub V \sub X$ are open affinoid subsets. We say that $U$ is relatively compact in $V$ over $Y$ if there are elements $f_1,\dots,f_r \in \oo_X(V)$ which topologically generate $\oo_X(V)$ over $\oo_Y(Y)$ and such that $|f_i(x)|<1$ for all $x\in U$.
\end{defi}

\begin{defi}\label{prop: rigid} Let $f : X \to Y$ be a separated morphism of rigid spaces over $K$.
\begin{enumerate}
\item We say that $f$ is partially proper if there exists an open affinoid cover $(Y_i)_{i\in I}$ and, for each $i \in I$, two open affinoid covers $(U_j)_{j\in J}$ and $(V_j)_{j\in J}$ of $f^{-1}(Y_i)$ such that $U_j \sub V_j$ and $U_j$ is relatively compact in $V_j$ over $Y_i$, for all $j$. We say that $X$ is partially proper if the structure map $X \to \Spa(K)$ is partially proper.

\item We say that $f$ is proper if $f$ is partially proper and quasicompact. We say that $X$ is proper if the structure map $X \to \Spa(K)$ is proper.
\end{enumerate}
\end{defi}

\begin{exer}\label{Kiehl proper implies proper}
Prove that if a map of rigid spaces $X \to Y$ over $K$ is (partially) proper in the sense of Definition \ref{prop: rigid}, then it is (partially) proper in the sense of Definition \ref{sep and prop: adic} (note that the $K$ appearing in Definition \ref{sep and prop: adic} is not the $K$ that $X$ and $Y$ are defined over, but completely arbitrary).
\end{exer}

This exercise points to one of the utilities of Definition \ref{prop: rigid} -- the ``rigid'' definition is often easier to check in practice. The converse to Exercise \ref{Kiehl proper implies proper} is known, but it is much more difficult and the proofs that one can (try to) piece together from the literature (as far as the author is aware of) use formal models as an intermediate. Let $f: X \to Y$ be a morphism of separated, quasi-paracompact rigid spaces. Then $f$ has a formal model $\mf{f}$ \cite[\S II.8.4, Lemma 4]{bosch} (see Lecture 3 below for more information on formal models), and Huber has shown that $\mf{f}$ is proper (in the sense that the induced morphism on special fibres is a proper morphism of schemes) if and only if $f$ is proper in the sense of Definition \ref{sep and prop: adic} \cite[Remark 1.3.18(ii)]{hu-etale}, and Temkin has shown $\mf{f}$ is proper if and only if $f$ is proper in the sense of Definition \ref{prop: rigid} \cite[Corollaries 4.4, 4.5]{temkin}. Using \cite[Remark 1.3.18(iii)]{hu-etale} and Temkin's results, we would optimistically believe that one can show that the two definitions of partial properness agree as well, but we have not checked the details.

At this point, let us state Kiehl's proper mapping theorem. For a proof, we refer to \cite[\S I.6.4]{bosch}.

\begin{theo}
Let $f : X \to Y$ be a proper map of rigid spaces over $K$ (in the sense of Definition \ref{prop: rigid}) and let $\mc{F}$ be a coherent $\oo_X$-module. Then, for any $i\geq 0$, the higher direct image $R^i f_\ast(\mc{F})$ is a coherent $\oo_Y$-module.
\end{theo}

In particular, when $Y=\Spa(K)$, the theorem says that $H^i(X,\mc{F})$ is a finite dimensional $K$-vector space. Our next topic in this lecture is the notion of (locally) quasifinite morphisms, with the goal of (mostly) proving Proposition \ref{characterization of locally quasifinite maps}, which plays a key role in the construction of eigenvarieties. 

\begin{defi}
Let $f : X \to Y$ be a morphism of adic spaces which is locally of finite type.

\begin{enumerate}
\item We say that $f$ is locally quasifinite if $f^{-1}(y)$ is a discrete topological space for every $y\in Y$.

\item We say that $f$ is quasifinite if it is locally quasifinite and quasicompact.
\end{enumerate} 
\end{defi}

Just as in algebraic geometry, the notion of quasifiniteness allows one to give a characterization of finite morphisms inside proper morphisms.

\begin{prop}\label{qf and proper implies finite}
Let $f : X \to Y$ be a morphism of adic spaces which is locally of finite type. Then $f$ is finite if and only if $f$ is proper and quasifinite.
\end{prop}

\begin{proof}
We will only sketch the proof in the easy direction that finite morphisms are proper and quasifinite; for the converse we refer to \cite[Proposition 1.5.5]{hu-etale}. We start with properness. Quasicompactness is immediate, so it suffices to prove partial properness. Properness is local on $Y$, so we may assume that $Y=\Spa(A,A^+)$ and $X=\Spa(B,B^+)$ are affinoid. Looking at the valuative criterion, we consider any diagram
\[
\xymatrix{
(A,A^+) \ar[r]^{\phi}\ar[d]^\eta & (B,B^+)  \ar[d]^\psi \\
(K,K^+) \ar[r] & (K,K^\circ)
}
\]
where $K$ is a nonarchimedean field and $K^+\sub K^\circ$ is an open valuation subring, and we need to show that there is a unique morphism $(B,B^+) \to (K,K^+)$ making the diagram commutes. Uniqueness is clear if it exists, since it has to be given by $\psi$, and existence is a question of whether $\psi(B^+) \sub K^+$ or not. But $\eta(A^+)\sub K^+$ and $B^+$ is the integral closure of $\phi(A^+)$, so since $K^+$ is integrally closed we see that $\psi(B^+)\sub K^+$ as desired.

It remains to prove that $f$ is locally quasifinite. Let $y\in Y$, corresponding to a map $\Spa(K,K^+) \to Y$. The adic space fibre product $\Spa(K,K^) \times_Y X$ is $\Spa(C,C^+)$, where $C = B \otimes_A K$ and $C^+$ is the integral closure of the image of $B^+\otimes_{A^+} K^+$ inside $C$. In particular, $(K,K^+) \to (C,C^+)$ is finite, and hence the nilreduction $C_{red}$ is a finite product of finite field extensions $C_i$ of $K$. Since any valuation only has finitely many extensions to a finite field extension and any two such are incomparable, it follows that $f^{-1}(y)$ is a finite discrete set.
\end{proof}

\begin{exer}
Let $f : X \to Y$ be a finite morphism of rigid spaces over $K$, where $K$ is a nonarchimedean field. Prove directly that $f$ is proper in the sense of Definition \ref{prop: rigid}. 
\end{exer}

We can now state a characterization, due to Huber \cite[Proposition 1.5.6]{hu-etale}, of morphisms that are locally quasifinite and partially proper. This characterization is important in the theory of eigenvarieties; see \cite[Th\'eor\`eme B.1]{aip-halo} and the lectures of Ludwig in this spring school \cite{ludwig-heid}. See also \cite[Proposition 1.5.4]{hu-etale} for a characterization of locally quasifinite morphisms, and \cite[Theorem A.1.3]{conrad-relative-ampleness} for the characterization of locally quasifinite morphisms in rigid analytic geometry. We also refer to \cite[\S 4]{buzzard-eigenvarieties}, where similar results are discussed (implicitly) in the rigid analytic setting. Before stating the characterization, we need to define the notion of tautness.

\begin{defi}\label{def of taut}
Let $X$ be an adic space. $X$ is taut if $X$ is quasiseparated and the closure of any quasicompact open subset of $X$ is quasicompact.
\end{defi}

Most rigid spaces (or indeed adic spaces) are taut, including the ones that we will apply the characterization to, and all quasiseparated rigid analytic curves \cite[Proposition 3.4.7]{achinger-lara-youcis}. For an example of a non-taut rigid space, see \cite[Example 8.3.8]{hu-etale}.

\begin{prop}\label{characterization of locally quasifinite maps}
Let $f: X \to Y$ be a morphism of adic spaces. Assume that $X$ is quasiseparated, that $Y$ is taut and that $f$ is locally of finite type. Then the following conditions are equivalent:
\begin{enumerate}
\item $f$ is locally quasifinite and partially proper;

\item For every $x\in X$ with image $y=f(x) \in Y$ there exists open subsets $U\sub X$, $V\sub Y$ such that $f(U) \sub V$, $\ol{\{ x\}} \sub U$, $\ol{\{y\}}\sub V$ and $f : U \to V$ is finite.
\end{enumerate}
\end{prop}

\begin{proof}
That (2) implies (1) is clear. The proof of the converse, given in general in \cite[Proposition 1.5.6]{hu-etale}, proceeds by first showing (or using) that the assumptions imply that $X$ is taut as well; see \cite[Lemma 1.3.13]{hu-etale}. We give the proof assuming that $X$ is taut, since this condition is usually easy to verify independently in applications. It suffices to prove the statement for rank one points, so let $x\in X$ be a rank one point. Put $y=f(x)$; this is also a rank one point. 

Our first goal is to find a quasicompact open subset $W \sub X$ such that $\ol{\{x\}} \sub W$ and such that no specialization of any point in $f^{-1}(y)\setminus \{x\}$ is contained in $\ol{W}$. We do this as follows. First, choose a quasicompact open subset $W_1$ such that $\ol{\{x\}} \sub W_1$ (can be done by Lemma \ref{some topology}). Since $f^{-1}(y)$ is discrete, it consists entirely of rank one points. In addition, $f^{-1}(y) \cap W_1$ is finite. Let $\Sigma$ denote the set of quasicompact open subsets in $W_1$ containing $\ol{\{x\}}$. By Lemma \ref{some topology}(3), $\ol{\{x\}} = \bigcap_{W^\prime \in \Sigma} W^\prime$. Set $S = (f^{-1}(y) \cap W_1) \setminus \{x\}$. Then we have $\ol{\{x\}} \cap S = \emptyset$, and from this it follows that there is a $W\in \Sigma$ such that $W \cap S = \emptyset$ (by the argument in the proof of Lemma \ref{some topology}(2) below). We have $\ol{\{x\}} \sub W$ by construction, so it remains to show that no specialization of any point in $f^{-1}(y)\setminus \{x\}$ is contained in $\ol{W}$. But $W \cap f^{-1}(y) = \{x\}$ by construction, and by Lemma \ref{some topology}(2) $\ol{W}$ is precisely the set of specializations of points in $W$. Since rank one generalizations are unique, we get what we want.

Now consider $y \in Y$. Let $\Delta$ be the set of quasicompact open subsets in $Y$ containing $\ol{\{y\}}$; by Lemma \ref{some topology}(4) we have $\ol{\{y\}} = \bigcap_{V^\prime\in \Delta}\ol{V^\prime}$. Since $f^{-1}(\ol{\{y\}})$ is the set of all specializations of elements of $f^{-1}(y)$, it follows that
\begin{equation}\label{equation in qf char proof}
\ol{\{x\}} = \bigcap_{V^\prime \in \Delta} \ol{f^{-1}(V^\prime)\cap W}.
\end{equation}
Now consider $f|_{W_2}: W_2 \to Y$, where $W_2 \sub X$ is a quasicompact open subset containing $\ol{W}$. By Lemma \ref{quasicompactness of maps} $f|_{W_2}$ is quasicompact, so $f^{-1}(V^\prime) \cap W$ is quasicompact for all $V^\prime \in \Delta$. From equation (\ref{equation in qf char proof}) and a short argument using the constructible topology on $W_2$ we can choose a $V\in \Delta$ such that $\ol{f^{-1}(V)\cap W} \sub W$. Set $U = f^{-1}(V)\cap W$; we claim that $f|_U : U \to V$ is finite, which would finish the proof.

To see this, we first note that it is quasifinite (it is locally quasifinite since $f$ is, and quasicompact by Lemma \ref{quasicompactness of maps}). By Proposition \ref{qf and proper implies finite}, it remains to show that $f|_U$ is partially proper. Consider a diagram
\[
\xymatrix{
\Spa(K,K^\circ) \ar[r]\ar[d] & U \ar[r]\ar[d]^{f|_U} & X  \ar[d]^f \\
\Spa(K,K^+) \ar[r] & V \ar[r] & Y,
}
\]
where $K$ is a nonarchimedean field and $K^+\sub K^\circ$ is an open valuation subring; we need to show that there is a unique $\Spa(K,K^+) \to U$ making the diagram commute. By partial properness of $f$, there is a unique $\Spa(K,K^+) \to X$ making the diagram commute and we have to show that it factors through $U$. Since the composition with $f$ factors through $V$,  $\Spa(K,K^+) \to X$ factors through $f^{-1}(V)$. Additionally, since $\Spa(K,K^\circ) \to X$ factors through $U$, $\Spa(K,K^+) \to X$ must factor through $\ol{U}$. Since $\ol{U} \sub W$, we conclude that $\Spa(K,K^+) \to X$ must factor through $f^{-1}(V) \cap W =U$, as desired. 
\end{proof}

Here we record some facts in general topology that where used in the proof above. Recall that a spectral space can be equipped with a finer topology, called its constructible topology, which is compact and Hausdorff. Note that Definition \ref{def of taut} is purely topological -- in general we say that a locally spectral space is taut if every quasicompact open subset has quasicompact closure. The proof of this Lemma is given in the appendix, see Lemma \ref{some topology proof}.

\begin{lemm}\label{some topology}
Let $X$ be a taut locally spectral space.
\begin{enumerate}
\item Let $S \sub X$ be a quasicompact subset. Then $\ol{S}$ is quasicompact. In particular, if $x\in X$, then $\ol{\{x\}}$ is quasicompact.

\item Let $S\sub X$ be a quasicompact subset. Then $x\in \ol{S}$ if and only if $x$ is a specialization of a point in $S$.

\item Assume that $X$ is the underlying topological space of an analytic adic space. Let $x \in X$ be a rank one point and let $\Sigma$ be the set of quasicompact open subsets containing $\ol{\{x\}}$. Then $\ol{\{x\}} = \bigcap_{U\in \Sigma} U$.

\item Assume that $X$ is the underlying topological space of an analytic adic space. Let $x \in X$ be a rank one point. Then $\ol{\{x\}} = \bigcap_{U\in \Sigma} \ol{U}$, where $\Sigma$ is as in (3).
\end{enumerate}
\end{lemm}

\begin{rema}\label{rigid spaces are almost locally compact}
Let us make some remarks on the role of the topological space of a rigid space (as it is defined here, i.e. as an adic space). In Tate's original theory, one does not have a topological space and the notion of coverings is instead rather subtle, making many arguments of a topological nature tricky (and many other arguments outright impossible). In an ideal world, one might have instead hoped for a theory that provides one with a locally compact topological space, as is the case for complex manifolds. The theory of adic spaces almost accomplishes this: The topological space we get is not locally compact, but it is very close. At first, if one is familiar with schemes, it might seem rather unsatisfying to obtain a locally spectral space, but the locally spectral spaces showing up as topological spaces of rigid spaces are very different from those showing up as topological spaces of algebraic varieties. If $X$ is a rigid space, let $X^B$ denote the subset of rank one points. Mapping an arbitrary point in $X$ to its unique rank one generalization, we get a map $X \to X^B$ and we equip $X^B$ with the quotient topology; this is the Berkovich space of $X$. When $X$ is taut, $X^B$ is a locally compact Hausdorff space. While one can often argue directly on $X$, it is sometimes helpful to use $X^B$ to supply intuition about the geometry of $X$. We encourage the reader to think about the proof of Proposition \ref{characterization of locally quasifinite maps} in these terms, for example. The Berkovich space $X^B$ will be discussed a bit more in Lecture 4.
\end{rema}

For completeness, let us also state the GAGA theorem, due to K\"opf \cite{kopf}; see also \cite{poineau-gaga,conrad-relative-ampleness}.

\begin{theo}[GAGA] Let $S$ be a proper algebraic variety over $K$, with analytification $S^{an}$. Given a coherent sheaf $\mc{F}$ on $S$, there is a functorial analytification $\mc{F}^{an}$, which is coherent sheaf. Then one has the following:
\begin{enumerate}
\item The functor $\mc{F} \mapsto \mc{F}^{an}$ is an equivalence of the abelian categories of coherent sheaves on $S$ and on $S^{an}$, respectively;

\item The two $\delta$-functors $\mc{F} \mapsto H^i(S,\mc{F})$ and $\mc{F} \mapsto H^i(S^{an},\mc{F}^{an})$ from coherent sheaves on $S$ to $K$-vector spaces are isomorphic.
\end{enumerate}
\end{theo}

\begin{rema}
There are also relative versions of the GAGA theorem; we refer to the references above.
\end{rema}

To finish this lecture, we discuss the notion of an affinoid morphism, which seems like it might be a natural notion in norchimedean geometry. In algebraic geometry, there is a robust notion of an affine morphism, but it turns out that the idea of an affinoid morphism is not so robust. Let $f: X\to Y$ be a morphism on (reasonable) adic spaces. There are two possible definitions of an affinoid morphism that one could use:

\begin{enumerate}
\item We could say that $f$ is affinoid if there is a cover $(U_i)_{i\in I}$ of $Y$ by open affinoids such that $f^{-1}(U_i)$ is affinoid for all $i\in I$.

\item We could also say that $f$ is affinoid if $f^{-1}(U)$ is affinoid for any open affinoid $U\sub Y$.
\end{enumerate}

Definition (2) is stronger than definition (1), but definition (1) is more practical (and suffices for many applications, such as vanishing of higher direct images of coherent sheaves). One might hope that the two definitions are equivalent (as is the case for the analogous definitions of affine morphisms in algebraic geometry). However, this is not the case, even in the world of rigid spaces. One particular counterexample is provided in \cite[Example 9.1.2]{berkeleylectures}. We present it below as an exercise.

\begin{exer} Let $K$ be a nonarchimedean field and set $X = \Spa K\langle x,y \rangle$. Define an open subset $V \sub X$ by 
\[
V = \{ (x,y) \in X \mid \,\,  |x|=1 \,\, \text{or} \,\, |y|=1 \}.
\]
Prove that $V$ is not affinoid by showing that $H^1(V,\oo_V)$ is nonzero. In particular, the inclusion $V\sub X$ is not affinoid in the sense of definition (2) above. However, show that the inclusion is affinoid in the sense of definition (1) above, by finding three rational subsets $U_0$, $U_1$ and $U_2$ covering $X$ such that the intersections $U_i \cap V$ are affinoid.
\end{exer}

\newpage

\section{Lecture 3: More on rigid spaces}\label{sec: lecture 3}

This lecture is a slightly disjointed collection of topics. We begin by defining flat, smooth and \'etale morphisms. Afterwards we discuss quasi-Stein spaces, which have many of the same properties as affinoid spaces and include the analytifications of all affine varieties. Finally, we will discuss Raynaud's approach to rigid geometry through formal models, with emphasis on the relation between the adic space of a rigid space and the collection of all its formal models.

\subsection{Flat, smooth and \'etale morphisms}

We start by discussing flat, smooth and \'etale morphisms. Huber defines flatness as follows (cf. \cite[Lemma 1.6.4(ii)]{hu-etale}):

\begin{defi}
Let $f : X \to Y$ be a morphism of adic spaces. We say that $f$ is flat if the induced homomorphisms $\oo_{Y,f(x)} \to \oo_{X,x}$ are flat for all $x\in X$.
\end{defi} 

This definition appears to be the direct parallel to the usual definition in algebraic geometry. However, one could also imagine a definition where $f$ is flat if, for any affinoid open $U=\Spa(B) \sub X$ mapping into an affinoid open $V=\Spa(A) \sub Y$, the induced ring map $A \to B$ is flat. To give this definition a name in the discussion that follows, let us call an $f$ satisfying this definition ``affinoid-flat''. In algebraic geometry, the analogous definitions are easily seen to be equivalent. In adic geometry the situation is murkier. When $X$ and $Y$ are locally strongly Notherian and Tate, flatness is known to imply affinoid-flatness \cite[Lemma B.4.3]{zavyalov-quotients}, but the converse is more subtle. See \cite[Lemma B.4.6]{zavyalov-quotients} for a converse under extra assumptions, which are satisfied in the case of rigid spaces. We will give the proof in the case of rigid spaces in a pair of exercises below, since it illustrates some of the ideas and techniques required for working with local rings of adic spaces, and the use of working with ``classical points'' (i.e. those corresponding to maximal ideals\footnote{Let $X$ be a rigid space over $K$. We say that a point $x\in X$ is classical if, for one (or, equivalently, any) open affinoid $U=\Spa A \sub X$ containing $x$ such that $x$ corresponds to a maximal ideal of $A$. This can also be phrased as saying that the residue field at $x$ is a finite extension of $K$.}) of rigid spaces even in the context of adic geometry, where one has many more points. This last point is analogous to how it often suffices (and can be advantageous) to work with spectra of maximal ideals in algebraic geometry over a field. 

To get off the ground, one needs the following important facts about local rings of rigid spaces at classical points \cite[\S I.4.1, Propositions 2 and 6]{bosch}:

\begin{prop}\label{local rings of rigid spaces}
If $Z = \Spa C$ is an affinoid rigid space over $K$ and $\m \sub C$ is a maximal ideal corresponding to a classical point $z \in Z$, then there is a natural injective map $C_{\m} \to \oo_{Z,z}$ which induces an isomorphism on completions. Moreover, $\oo_{Z,z}$ is Noetherian.
\end{prop}

In fact, $\oo_{Z,z}$ is even excellent \cite[Theorem 1.1.3]{conrad-irreducible}. The two exercises below then show that flatness and affinoid-flatness are equivalent for rigid spaces. For these exercises, the reader may wish to look at the algebra results recalled in Appendix \ref{sec: appalgebra}.

\begin{exer}\label{ex: flatness1} Let $f : X=\Spa(B) \to Y=\Spa(A)$ be a morphism of affinoid rigid spaces over $K$. Show that $A\to B$ is flat if and only if the corresponding map $\oo_{Y,f(x)} \to \oo_{X,x}$ is flat for every \emph{classical} point $x\in X$.
\end{exer}

\begin{exer}\label{ex: flatness2} Let $f : X \to Y$ be a morphism of affinoid rigid spaces over $K$. Show that $f$ is flat if and only if the corresponding map $\oo(Y) \to \oo(X)$ is flat. Deduce that a general map of rigid spaces is flat if and only if it is affinoid-flat.
\end{exer}

For the definitions of smooth and \'etale morphisms, we use an infinitesimal lifting criterion analogous to the one in algebraic geometry. We refer to \cite[\S 1.6]{hu-etale} for more details and results, including characterizations in terms of differentials.

\begin{defi}
Let $f : X \to Y$ be a morphism of locally Noetherian adic spaces. Assume that $f$ is of locally finite type. Assume that $(A,A^+)$ is a Noetherian Huber pair and that $I\sub A$ is an ideal with $I^2 = 0$. Write $(A/I)^+$ for the integral closure of $A^+$ inside $A/I$ and write $S = \Spa(A,A^+)$ and $T=\Spa(A/I,(A/I)^+)$.
\begin{enumerate}
\item We say that $f$ is smooth if, for any $(A,A^+)$ and $I$ as above and any morphism $S \to Y$, any $Y$-morphism $T \to X$ lifts to a $Y$-morphism $S \to X$.

\item We say that $f$ is \'etale if, for any $(A,A^+)$ and $I$ as above and any morphism $S \to Y$, any $Y$-morphism $T \to X$ lifts uniquely to a $Y$-morphism $S \to X$.
\end{enumerate}
\end{defi}

In general, the behaviour of (analytic) adic spaces is somewhere between the world of schemes and the world of complex analytic spaces, as indicated in Remark \ref{rigid spaces are almost locally compact}. We record a few such facts here. The first is \cite[Example 1.6.6(ii), Lemma 2.2.8]{hu-etale}, and was used by Scholze as the definition of \'etale morphisms in the theory of perfectoid spaces.

\begin{prop}
Let $f : X \to Y$ be a locally finite type morphism of locally Noetherian adic spaces.
\begin{enumerate}
\item  $f$ is finite and \'etale if and only if there is a cover $(U_i)_{i\in I}$ of $Y$ such that, for each $i$, $U_i$ and $f^{-1}(U_i)$ are affinoid and $\oo(U_i) \to \oo(f^{-1}(U_i))$ is finite \'etale. If this holds, then for every affinoid $U \sub Y$ we have that $\oo(U) \to \oo(f^{-1}(U))$ is finite \'etale.

\item $f$ is \'etale if and only if, for any $y\in Y$, there is an open $U\sub Y$ containing $y$ and a finite \'etale morphism $g : V \to U$ such that $f^{-1}(U) \to U$ factors as $g \circ j$, where $j : f^{-1}(U) \to V$ is an open immersion.
\end{enumerate}
\end{prop}

We note that the analogue of part (2) fails in the world of schemes (and encourage the interested reader to try to show this). In some sense, part (2) means that \'etale morphisms of analytic adic spaces are closer to being local isomorphisms than \'etale morphisms of algebraic varieties. Another fact pointing in this direction is the following.

\begin{exer}
Let $f : X=\Spa(A) \to Y=\Spa(B)$ be an \'etale map of affinoid rigid spaces over $K$. Assume that $x\in X$ is a classical point such that $x$ and $f(x)$ have the same residue field. Then there are open sets $x\in V \sub X$ and $f(x) \in U \sub Y$ such that $f : V \to U$ is an isomorphism.
\end{exer}

We remark that the assumption that $x$ is classical and has the same residue field as $f(x)$ is crucial. In particular, even if $K$ is algebraically closed, an \'etale map need not be a local isomorphism at all points, though it is around classical points. We also state the following characterization of smooth morphisms, see \cite[Corollary 1.6.10]{hu-etale}, saying roughly that smooth morphisms locally look like fibrations.

\begin{prop}
Let $f : X \to \Spa A$ be a morphisms of locally Noetherian adic spaces with $A$ Noetherian. Then $f$ is smooth if and only if, for every $x\in X$, there exists an open subset $U\sub X$ containing $x$ and a factorization $U \overset{j}{\to} \Spa A \langle X_1,\dots,X_n \rangle \to \Spa A $ for some $n$ (depending on $x$), where $j$ is \'etale.
\end{prop}

\begin{rema}
We have stated this proposition in the case $A^+ = A^\circ$ for notational convenience only.
\end{rema}

\subsection{Quasi-Stein spaces}

A type of rigid space which frequently shows up is a quasi-Stein space.

\begin{defi}\label{def of quasi-stein space}
Let $X$ be a rigid space over $K$. We say that $X$ is a quasi-Stein space if there is an open cover $X= \bigcup_{n=1}^\infty U_n$ satisfying the following properties:
\begin{enumerate}
\item Each $U_n$ is affinoid and $U_n \sub U_{n+1}$ for all $n\geq 1$;

\item The restriction maps $\oo_X(U_{n+1}) \to \oo_X(U_n)$ have dense image for all $n\geq 1$.
\end{enumerate}

If, in addition, $U_n$ is relatively compact inside $U_{n+1}$ for all $n\geq 1$, then we say that $X$ is a Stein space.
\end{defi}

Note that a Stein space is partially proper by definition (according to Definition \ref{prop: rigid}). The simplest examples of quasi-Stein spaces are affinoid spaces (just take $U_n=X$ for all $n$). On the other hand, there are plenty of Stein spaces coming from affine varieties: If $S$ is an affine variety over $K$, written explicitly as 
\[
S = \Spec K[X_1,\dots,X_d]/(f_1,\dots,f_r),
\]
then its analytification $S^{an}$ is defined by
\[
S^{an} = \bigcup_{n=0}^\infty \Spa K\langle \pi^n X_1 \dots, \pi^n X_d \rangle / (f_1,\dots,f_r),
\]
where $\pi$ is a pseudouniformizer in $K$, and this description is functorial in $S$ and glues (and is independent of the choice of $\pi$). 

\begin{exer}
Check that $S^{an}$ is a Stein space when $S$ has no component of dimension $0$.
\end{exer}

In particular, all analytifications of affine varieties (indeed of all varieties) are partially proper, giving a large supply of partially proper rigid spaces. Another example of a Stein space is the open unit disc $D$, which we may define as
\[
D = \bigcup_{n=1}^\infty \Spa K\langle \pi^{-1}T^n \rangle \sub \Spa K\langle T \rangle.
\]
General quasi-Stein spaces behave like affinoids in the following sense:

\begin{theo}\label{Thm A and B for quasistein spaces}
Let $X$ be a quasi-Stein space over $K$ with a cover $X = \bigcup_{n=1}^\infty U_n$ as in Definition \ref{def of quasi-stein space}, and let $\mc{F}$ be a coherent sheaf on $X$.
\begin{enumerate}
\item We have $H^i(X,\mc{F})=0$ for all $i\geq 1$.

\item The image of $\mc{F}(X)$ is dense in $\mc{F}(U_n)$ for all $n$, and in fact $\mc{F}(U_n) \cong \mc{F}(X) \otimes_{\oo_X(X)} \oo_X(U_n)$ for all $n$.
\end{enumerate}
\end{theo}

In particular, coherent sheaf theory on $X$ is entirely controlled by the modules of global sections. We remark that if $\mc{F}$ is a coherent sheaf of $X$ then $\mc{F}(X)$ need not be a finitely generated $\oo_X(X)$-module. The $\oo_X(X)$-modules arising as $\mc{F}(X)$ for some coherent sheaf are called \emph{coadmissible}. For a reference for Theorem \ref{Thm A and B for quasistein spaces} (in the more general setting of Fr\'echet-Stein algebras; see below), see \cite[\S 3, Theorem, Corollary 3.1]{schneider-teitelbaum-dist}.

\begin{rema}
One might wonder to which extent vanishing of higher cohomology of coherent sheaves characterizes quasi-Stein spaces. In algebraic geometry, a quasicompact and quasiseparated scheme $X$ is affine if and only if $H^i(X,\mc{F})=0$ for all quasicoherent $\mc{F}$ and all $i\geq 1$. In complex geometry, a complex analytic space is Stein if and only if $H^i(X,\mc{F})=0$ for all coherent $\mc{F}$ and all $i\geq 1$. However, the analogue is not true in rigid geometry. We refer to \cite{liu,maculan-poineau-stein} for more information.
\end{rema}

In particular, we see that the theory of coherent sheaves over a quasi-Stein space $X$ can be described entirely in terms of coadmissible $\oo(X)$-modules. Moreover, $\oo(X)$ itself is the inverse limit $\varprojlim_n \oo(U_n)$, and the transition maps are flat. A generalization of the rings $\oo(X)$ and their coadmissible modules was given by Schneider and Teitelbaum \cite[\S 3]{schneider-teitelbaum-dist}, under the name of Fr\'echet--Stein algebras. This theory is very important for (locally analytic) representation theory of $p$-adic groups, and featured in Hellmann's lectures in the spring school. Here, we give the basic definitions.

\begin{defi} Let $K$ be a nonarchimedean field and let $A$ be a (not necessarily commutative) $K$-algebra.
\begin{enumerate}
\item We say that $A$ is a Fr\'echet--Stein algebra if we can write $A = \varprojlim_{n\geq 1} A_n$ where the $A_n$ are left Noetherian $K$-Banach algebras, the transition maps $A_{n+1} \to A_n$ are injective of operator norm $\leq 1$ and make $A_n$ into a flat right $A_{n+1}$-module, and the maps $A \to A_n$ have dense image.

\item If $A = \varprojlim_n A_n$ is a Fr\'echet--Stein algebra, then a coherent sheaf for $A$ is an inverse system $(M_n)_n$ of finitely generated left $A_n$-modules $M_n$, such that the transition maps induce isomorphisms $A_n \otimes_{A_{n+1}} M_{n+1} \cong M_n$.

\item If $A = \varprojlim_n A_n$ is a Fr\'echet--Stein algebra, then a left $A$-module $M$ is called coadmissible if there exists a coherent sheaf $(M_n)_n$ with $M \cong \varprojlim_n M_n$. 
\end{enumerate}

\end{defi}

The standard references for Fr\'echet--Stein algebras are the original paper \cite{schneider-teitelbaum-dist} and Emerton's monograph \cite{emerton-locallyanalytic}. We remark that, if $A =\varprojlim_n A_n$ is a Fr\'echet--Stein algebra, then the norms on the $A_n$ induce norms on $A$ and $A$ is a Fr\'echet space (indeed a Fr\'echet algebra) with respect to these norms (hence the ``Fr\'echet'' in the terminology). We also note that being a Fr\'echet--Stein algebra is an intrinsic property of a Fr\'echet algebra, as is being a coadmissible module; see \cite[Propositions 1.2.7 and 1.2.9]{emerton-locallyanalytic}. We state the analogue of Theorem \ref{Thm A and B for quasistein spaces} for Fr\'echet--Stein algebras (see \cite[\S 3, Theorem]{schneider-teitelbaum-dist}).

\begin{theo}
Let $A =\varprojlim_n A_n$ be a Fr\'echet--Stein algebra and let $M = \varprojlim_n M_n$ be a coadmissible module.
\begin{enumerate}
\item The natural map $M \to M_n$ has dense image.

\item The higher derived functor $R^i\varprojlim_n M_n$ vanish for all $i\geq 1$.
\end{enumerate}
\end{theo}

Our final topic on quasi-Stein spaces will be a version of Serre duality, due to Chiarellotto \cite{chiarellotto}, which was used in Hellmann's lectures in the spring school. It is restricted to smooth Stein spaces, and involves a notion of cohomology with compact support for sheaves, which we will now define. Chiarellotto works with Tate's version rigid geometry, and is therefore forced to work with the topos associated with a rigid space in the sense of Tate. Since this topos is the topos associated with the (topological space of the) adic space, the definition is simpler in the language of adic spaces.

\begin{defi}
Let $X$ be a taut rigid space and let $\Sigma$ be the set of quasicompact closed subsets of $X$.
\begin{enumerate}
\item Let $Z\in \Sigma$ and let $\mc{F}$ be a sheaf on $X$. We define $H^0_Z(X,\mc{F}) \sub \mc{F}(X)$ to be the subgroup of sections whose support is contained in $Z$.

\item Let $Z\in \Sigma$ and let $\mc{F}$ and $\mc{G}$ be sheaves on $X$. We denote the relative Hom-sheaf of $\mc{F}$ and $\mc{G}$ by $\shom(\mc{F},\mc{G})$, and we define $\Hom_Z(\mc{F},\mc{G}):=H_Z^0(\shom(\mc{F},\mc{G}))$.

\item Let $Z\in \Sigma$. We define $H^i_Z(X,-)$ and $\Ext^i_Z(-,-)$ to be the derived functors of $H^0_Z(X,-)$ and $\Hom_Z(-,-)$, respectively.

\item We define $H^i_c(X,-)=\varinjlim_{Z\in \Sigma} H^i_Z(X,-)$ and $\Ext^i_c(-,-)=\varinjlim_{Z\in \Sigma} \Ext^i_Z(-,-)$.

\end{enumerate}
\end{defi}

We can now state Chiarellotto's Serre duality theorem \cite[Theorem 4.21]{chiarellotto}:

\begin{theo}
Let $X$ be a smooth connected Stein space over $K$ of dimension $d$ and let $\mc{F}$ be a coherent sheaf on $X$. Then there are perfect pairings 
\[
H_c^{i}(X,\mc{F}) \times \Ext^{d-i}(\mc{F},\Omega_X^d) \to K
\]
and
\[
H^i(X,\mc{F}) \times \Ext_c^{d-i}(\mc{F},\Omega_X^d) \to K
\]
for all $i$. For appropriately defined topologies, these are pairings of topological vector spaces, making one space the strong dual of the other and vice versa.
\end{theo}

\subsection{Adic spaces and formal models}

Let $X$ be a quasicompact and quasiseparated rigid space over a nonarchimedean field $K$. We end this lecture by the discussing the relationship between $X$ and the collection of its formal models. For an introduction to the theory of formal schemes in the situation we need, as well as a detailed account of the relationship between rigid and formal geometry, we refer to \cite[Part II]{bosch}. While we will refer to that source for most details, we do start by recalling the notion of an admissible formal scheme over $K^\circ$.

\begin{defi} Let $\pi  \in K$ be a pseudouniformizer.
\begin{enumerate}
\item We say that a topological $K^\circ$-algebra is admissible if it has the $\pi$-adic topology, is flat over $K^\circ$ and isomorphic to $K^\circ \langle X_1,\dots,X_n \rangle /(f_1,\dots,f_r)$ for some $n$ and some $f_1,\dots,f_r \in K^\circ \langle X_1,\dots,X_n \rangle$.

\item We say that a formal $K^\circ$-scheme $\mf{X}$ is admissible if it is locally of the form $\Spf A$ with $A$ an admissible $K^\circ$-algebra.
\end{enumerate}

\end{defi}

There is a functor from quasiparacompact\footnote{A topological space is quasiparacompact if it admits a cover $(U_i)_{i\in I}$ by quasicompact open subsets such that, for every $i$, $U_i \cap U_j = \emptyset$ for all but finitely many $j$.} formal schemes over $K^\circ$ to quasiparacompact and quasiseparated rigid spaces over $K$. On affinoids, it is defined by
\[
\Spf A \mapsto \Spa A[1/\pi];
\]
the general case is obtained by gluing \cite[p. 171-172]{bosch}. Given a formal scheme $\mf{X}$ we write $\mf{X}^{rig}$ for its associated rigid space; it is called the (rigid) generic fibre of $\mf{X}$. Conversely, given a rigid space $X$, an admissible formal scheme $\mf{X}$ together with an isomorphism $f_{\mf{X}}: \mf{X}^{rig} \cong X$ is called a formal model of $X$ (the isomorphism $f_{\mf{X}}$ is usually suppressed from the notation). In such a situation, we have a continuous and closed surjective morphism
\[
\spe_{\mf{X}} : X \to \mf{X},
\]
called the specialization morphism, which naturally extends to a morphism of locally ringed spaces $(X,\oo_X^+) \to (\mf{X},\oo_{\mf{X}})$. At the level of affinoids $\mf{X}=\Spf A$, it sends a valuation $x\in X=\mf{X}^{rig}$ to the open prime ideal $
\spe_{\mf{X}}(x) = \{ f\in A \mid |f(x)| < 1 \}$
in $A$. 

\begin{rema}\label{formal schemes and adic spaces}
Formal schemes may be viewed as adic spaces. At the level of affines this simply replaces $\Spf A$ with $\Spa A$; see \cite[\S 4]{hu-generalization}. The set $\Spf A$ embeds into $\Spa A$ by sending an open prime ideal $\p$ to the valuation $x$ on $A$ with value group $\{1\}$ characterized by $|f(x)| = 0$ if and only if $f\in \p$. In this language, the specialization map may be interpreted purely in terms of valuations. We will come back to this viewpoint towards the end of the lecture.
\end{rema}

To compare admissible formal schemes and rigid spaces, one needs to study a particular type of morphisms between admissible formal schemes called admissible formal blowups. Given any admissible formal scheme $\mf{X}$ and a coherent open ideal sheaf $\mf{a} \sub \oo_{\mf{X}}$, one associates an admissible formal scheme $\varphi : \mf{X}_{\mf{a}} \to \mf{X}$ over $\mf{X}$ called the admissible formal blowup of $\mf{X}$ in $\mf{a}$; see \cite[\S II.8.2]{bosch}. The induced map $\varphi^{rig} : \mf{X}_{\mf{a}}^{rig} \to \mf{X}^{rig}$ is an isomorphism. With this setup, Raynaud's characterization of rigid spaces in terms of formal schemes is the following \cite[\S II.8.4, Theorem 3]{bosch}:

\begin{theo}\label{raynauds theorem}
The functor $\mf{X} \to \mf{X}^{rig}$ induces an equivalence between the category of quasiparacompact admissible formal $K^\circ$-schemes, localized at the class of admissible formal blowups, and the category of quasiparacompact and quasiseparated rigid spaces.
\end{theo}

In the course of the proof, one proves the following more concrete results \cite[\S II.8.4, Lemma 4, Lemma 5]{bosch}:

\begin{prop}\label{properties of formal models} The following statements hold:
\begin{enumerate}
\item Two morphisms $\varphi,\psi : \mf{X}_2 \to \mf{X}_1$ of admissible formal schemes are equal if and only if $\varphi^{rig}=\psi^{rig}$.

\item Let $\mf{X}_1$ and $\mf{X}_2$ be quasiparacompact admissible formal schemes and let $\sigma : \mf{X}_2^{rig} \to \mf{X}^{rig}_1$ be a morphism. Then there exists an admissible formal blowup $\varphi : \mf{X}_3 \to \mf{X}_2$ and a morphism $\psi : \mf{X}_3 \to \mf{X}_1$ such that $\psi^{rig} = \sigma \circ \varphi^{rig}$. 

\item Let $\mf{X}$ be a quasiparacompact admissible formal scheme and let $(U_i)_{i\in I}$ be a cover of $\mf{X}^{rig}$ by quasicompact opens such that, for every $i$, $U_i \cap U_j = \emptyset$ for all but finitely many $j$. Then there exists an admissible formal blowup $\varphi : \mf{X}_0 \to \mf{X}$ with a cover $(\mf{U}_i)_{i\in I}$ by quasicompact opens such that $\mf{U}_i^{rig}$ is identified with $U_i$ for all $i$ under the isomorphism $\varphi^{rig}$.

\item Any quasiparacompact and quasiseparated rigid space admits a formal model. 
\end{enumerate}
\end{prop}

We now wish to explain the perspective on the topological space of a quasicompact and quasiseparated rigid space $X$ offered by the theory of formal models. Let $(\mf{X}_i)_{i\in I}$ be the category of formal models of $X$. More precisely, the objects are pairs $(\mf{X}_i,f_i)$ where $\mf{X}_i$ is an admissible formal scheme and $f_i : \mf{X}_i^{rig} \to X$ is an isomorphism ($I$ is simply an indexing set, included for notational convenience). A morphism $\varphi : (\mf{X}_i,f_i) \to (\mf{X}_j,f_j)$ is a morphism $\varphi : \mf{X}_i \to \mf{X}_j$ satisfying $f_j \circ \varphi^{rig} = f_i$. 

\begin{exer}
Prove (using Proposition \ref{properties of formal models}) that the category of formal models is an inverse system. In particular, there is at most one morphism $\varphi : (\mf{X}_i,f_i) \to (\mf{X}_j,f_j)$ for any $i,j\in I$.
\end{exer}

With this observation, we can form the inverse limit $\varprojlim_i \mf{X}_i$ and take the inverse limit of the specialization maps
\[
\spe =(\spe_i)_{i\in I} : X \to \varprojlim_i \mf{X}_i.
\]
Here we have written $\spe_i := \spe_{\mf{X}_i}$. We then have the following theorem (cf. \cite[Theorem 2.22]{scholze-perfectoid}):.

\begin{theo}
The map $\spe$ is a homeomorphism, and extends to an isomorphism of locally ringed spaces $(X,\oo_X^+) \cong \varprojlim_i (\mf{X}_i, \oo_{\mf{X}_i})$.
\end{theo}

\begin{proof}
We prove that $\spe$ is a homeomorphism; the remaining part is left as an exercise (see below). For surjectivity, consider $x = (x_i)_i \in \varprojlim_i \mf{X}_i$. Put $X_i = \spe^{-1}_i(x_i)$; then $\spe^{-1}(x) = \bigcap_i X_i$. The sets $X_i$ are closed in the constructible topology and nonempty, so $\spe^{-1}(x)$ is nonempty and hence $\spe$ is surjective. For injectivity, let $x,y \in X$. Since $X$ is a $T_0$-space, without loss of generality we can choose a quasicompact open $U\sub X$ with $x\in U$, $y\notin U$. By Proposition \ref{properties of formal models}(3) there exists an $i$ and a quasicompact open $\mf{U} \sub \mf{X}_i$ such that $U = \spe^{-1}_i(\mf{U})$. Then $\spe_i(x) \in \mf{U}$ and $\spe_i(y) \notin \mf{U}$, so $\spe(x) \neq \spe(y)$ as desired. Finally, it remains to prove that $\spe$ is open. Let $\pi_i : \varprojlim_i \mf{X}_i \to \mf{X}_i$ be the projection. Let $U\sub X$ be a quasicompact open and let $i$ and $\mf{U} \sub \mf{X}_i$ be as in the proof of injectivity. Then $\spe_i = \pi_i \circ \spe$ and hence $\spe(U) = \spe(\spe_i^{-1}(\mf{U})) = \pi^{-1}(\mf{U})$, which is open by definition of the inverse limit topology. Therefore $\spe$ is open, finishing the proof that $\spe$ is a homeomorphism.
\end{proof} 

\begin{exer}\label{sp as iso of ringed spaces}
Show that $\spe$ extends to an isomorphism $(X,\oo_X^+) \cong \varprojlim_i (\mf{X}_i, \oo_{\mf{X}_i})$ of locally ringed spaces. Does the theorem extend to quasiparacompact and quasiseparated rigid spaces?
\end{exer}

One may view this as saying that formal models provide ``finite approximations'' to rigid spaces. This is often desirable: The adic space contains a huge amount of information, and while this is sometimes an advantage, it can also make them unwieldy. Formal models can provide an excellent way to analyse a rigid space. We will see an example of this in the next lecture when we discuss uniformization of curves. 

Let us discuss a few more aspects of the equivalence in Theorem \ref{raynauds theorem}. Proposition \ref{properties of formal models} asserts, in particular, that any morphism $f : X\to Y$ of quasiparacompact and quasiseparated rigid spaces has a formal model $\mf{f} : \mf{X} \to \mf{Y}$. A basic question is then whether $\mf{f}$ inherits properties of $f$, or at least can be made to do so after changing $\mf{f}$ to a different formal model ($\mf{f}$ is certainly not unique). Let us give a few key examples of how this can work out: 

\begin{itemize}
\item As mentioned in Lecture 2, properness of $\mf{f}$ is equivalent to properness of $f$.

\item If $f$ is \'etale (or smooth), then it can happen that there is no formal model $\mf{f}$ which is \'etale (or smooth). Indeed, this is typically the case (for a simple, zero-dimensional example, consider a ramified extension of local fields). This should not be alarming and hopefully feel familiar. For example, the reader familiar with elliptic curves over local fields will know that not all elliptic curves have even potentially good reduction. In particular, we cannot make the problem go any by extending the base field. We will see this implicitly in the next lecture when we discuss the Tate curve.

\item If $f$ is flat, then it is a remarkable and difficult theorem that $\mf{f}$ can be chosen to be flat. See \cite[\S II.9.4, Theorem 1]{bosch} and the references there.
\end{itemize}

Let us also say a few more words about formal schemes as adic spaces, adding to Remark \ref{formal schemes and adic spaces}. As mentioned there, an affine formal scheme $\Spf A$ with $A$ of topologically finite type over $K^\circ$ defines an adic space $\Spa A$ ($A$ is sheafy by \cite[Theorem 2.2]{hu-generalization} in the Noetherian case and \cite[Corollary 1.2]{zavyalov-sheafiness} in general), and this procedure glues to general formal schemes locally of topologically finite type over $K^\circ$. Given such a formal scheme $\mf{X}$, let us denote the associated adic space by $\mf{X}^{ad}$. Then the generic fibre $\mf{X}^{rig}$ of $\mf{X}$ introduced above is just the open subset $\pi \neq 0$ of $\mf{X}^{ad}$. In particular, in the world of adic spaces, the generic fibre is just the generic fibre in the same way as in the world of schemes. These remarks extend to the more general generic fibre functor introduced by Berthelot \cite[\S 7]{dejong}.

We finish this lecture by illustrating the theory above through an example, the closed unit disc. Though simple, this example is still very rich and illustrates many of the general features of theory. We refer to Bergdall's lectures \cite[\S 4]{bergdall-heid} for a thorough account of the closed unit disc. 

\begin{exer}\label{unit disc and formal models}
Consider the formal scheme $\mf{X} = \Spf K^\circ \langle T \rangle$, which is a formal model for the closed unit disc $X = \Spa K\langle T \rangle$. Blowing up $\mf{X}$ in $\mf{a} = (\pi,T-a)$, for $a\in K^\circ$, creates a new formal model of $X$ whose special fibre has two irreducible components, one isomorphic to $\A^1_k$ and one isomorphic to $\mb{P}^1_k$. One may iterate this to produce a cofinal set of formal models of $X$. Compare this to the description of the topological space of $X$ below and from Bergdall's lectures \cite[\S 4]{bergdall-heid}.  
\end{exer}

For the reader's convenience, and to introduce how we will draw schematic pictures of adic spaces in the next lecture, we include the following sketch of the topological space of $X = \Spa K\langle T \rangle$.

\vspace{0.5 cm}

\begin{figure}[htp]
\centering
\includegraphics{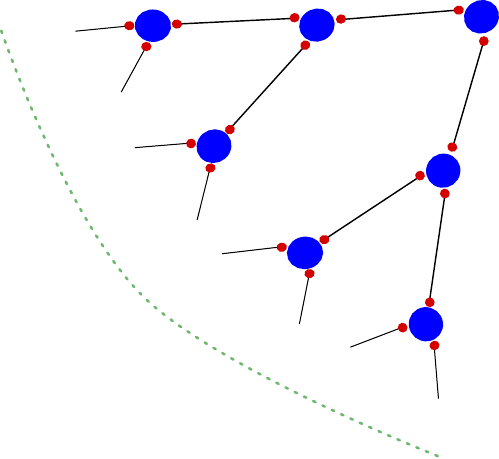}
\caption{A drawing of the closed unit disc $\Spa K\langle T \rangle$.}
\end{figure}

The reader may compare this to \cite[Example 2.20]{scholze-perfectoid}; we use the same terminology for the different types of points on $X$ (though we will ignore points of type 3 and 4). For natural reasons, the picture is a discretized version of the reality. Blue points are type $2$ points (rank $1$ points attached to closed discs inside $X$ of positive radius) and red points are the associated type $5$ points (of rank $2$). The black lines symbolize connective segments (which themselves, in reality, contain both type $2$ and type $5$ points and branch infinitely). The green points are the ends of the tree; these are the classical points from Tate's rigid geometry, corresponding to the maximal ideals of $K\langle T \rangle$, also called type $1$ points. In future pictures we will not include the type $1$ points. Finally, we encourage the interested reader to try to understand the local rings and residue fields of the various point of the closed unit disc, and to relate them to the formal models described in Exercise \ref{unit disc and formal models}.

\newpage

\section{Lecture 4: Uniformization of curves}

In this lecture, we will talk about uniformization, in the topological sense, of rigid spaces. Recall that a map $f : X \to Y$ of topological spaces is called a covering map if, for every $y\in Y$, there exists an open neighbourhood $V$ of $y$ such that $f^{-1}(V)$ is a disjoint union $\bigsqcup_{i\in I} U_i$ of open subsets $U_i \sub X$ such that $f|_{U_i} : U_i \to V$ is a homeomorphism for all $i$. We make the following definition.

\begin{defi}
We say that a map $f : X \to Y$ of (general) adic spaces is a covering map if for every $y\in Y$, there exists an open neighbourhood $V$ of $y$ such that $f^{-1}(V)$ is a disjoint union $\bigsqcup_{i\in I} U_i$ of open subsets $U_i \sub X$ such that $f|_{U_i} : U_i \to V$ is an isomorphism of adic spaces for all $i$.
\end{defi}

Note that if $f : X \to Y$ is a covering map of adic spaces, then $|f| : |X| \to |Y|$ is a covering map of topological spaces. Conversely, one has the following observation:

\begin{exer}
Let $Y$ be an adic space and let $g : S \to |Y|$ be a covering map of topological spaces. Show that there is a unique covering map $ f : X \to Y$ of adic spaces such that $|X|=S$ and $|f|=g$.
\end{exer}

The converse is slightly more subtle. For our purposes, the following observation suffices.

\begin{exer}
Assume that $X$ is an adic space with the action of a group $G$ such that $G$ acts freely and properly discontinuously on $|X|$. Let $T=|X|/G$ be the quotient and let $f : |X| \to T$ be the quotient map. Show that there is a unique covering map $f : X\to Y$ of adic spaces such that $|Y|=T$ and $f=|g|$.
\end{exer} 

We will use these observations to conflate coverings of an adic space and coverings of its underlying topological space, and we will do so without further comment. In algebraic geometry, the topological notion of a covering space is not very helpful. For example, a map between irreducible varieties is a covering map if and only if it is an isomorphism. This lead to the introduction of \'etale maps as a substitute. On the other hand, the notion of \'etale coverings is superfluous in the context of complex manifolds, as \'etale coverings are automatically topological coverings. In nonarchimedean geometry, the situation is somewhere in between. In particular, adic spaces have rich enough topology to allow for interesting examples of topological covering maps. We will explore this for smooth projective curves, starting with the genesis of rigid geometry, the Tate curve. To begin with, our goal is to give examples of covering maps which arise as quotients by discrete groups, and after that we discuss how to construct covering maps starting from a curve.

To construct Tate curves, consider the rigid space $\Gm$ over a nonarchimedean field $K$, which is the analytification of the usual scheme $\Gm = \Spec K[T,T^{-1}]$. We will never again use the algebraic $\Gm$, so we will simply write $\Gm$ for what we might otherwise denote $\Gm^{an}$. Explicitly, it is the union
\[
\Gm = \bigcup_{n=0}^\infty \Spa K\langle \pi^n T, \pi^n T^{-1} \rangle
\]
where $\pi \in K$ is a pseudouniformizer. Let $q \in K$ satisfy $0< |q| < 1$. The maps $T \mapsto q^m T$ define automorphisms of $\Gm$, giving a group homomorphism $\Z \to \Aut(\Gm)$. 

\begin{prop}
The action of $\Z$ on $\Gm$ is free and properly discontinuous, so we may take the quotient. We call the quotient the Tate curve (with parameter $q$), and denote it by $E_q$. Moreover, $E_q$ is proper and connected. 
\end{prop}

\begin{proof}
We need to show that every $x \in \Gm$ has a neighbourhood $U$ such that $m\cdot U \cap U \neq \emptyset$ for $m\in \Z$ implies that $m=0$. Consider the covering of $\Gm$ by open affinoid subsets $U_n = \{ |q|^{(n+1)/2} \leq |T| \leq |q|^{n/2} \}$, $n\in \Z$. Note that $U_n \cap U_k \neq \emptyset$ if and only $|n-k|\leq 1$. Choose $n$ such that $x\in U_n$. Then
\[
m\cdot U_n \cap U_n = U_{2m+n} \cap U_n,
\]
and these intersect if and only if $|(2m+n)-n|\leq 1$, i.e. $m=0$, as desired. Since $\Gm$ is connected, it is clear that $E_q$ is connected. It remains to prove that $E_q$ is proper. Let $\pi : \Gm \to E_q$ denote the quotient map. Quasicompactness is clear since $E_q$ is covered by the two affinoid opens $\pi(U_0)$ and $\pi(U_1)$. Moreover, consider the affinoid opens
\[
V_0 = \{ |q|^{3/5} \leq |T| \leq |q|^{-1/5}, \,\,\,\, V_1 = \{ |q|^{6/5} \leq |T| \leq |q|^{2/5} \}
\]
inside $\Gm$. They are mapped isomorphically into $E_q$ and $U_i \sub V_i$ is relatively compact, so $E_q$ is proper (we leave it as an exercise to check that $E_q$ is separated).
\end{proof}

We also have a group structure of $E_q$.

\begin{lemm}
$E_q$ inherits a commutative group structure from $\Gm$.
\end{lemm}

\begin{proof}
We use the functorial perspective. Let $\mathrm{Ab}$ be the category of abelian groups and let $\mc{C}$ be the category of rigid spaces over $K$, viewed as a site with respect to the analytic topology (i.e. $(U_i \to X)_{i\in I}$ is a cover of $X$ if the maps $U_i \to X$ are open immersions and $X = \bigcup_{i\in I} U_i$). Then $\Gm$ represents the functor $\mc{C} \to \mathrm{Ab}$ defined by $X \mapsto \oo(X)^\times$. We claim that $E$ represents the functor $\mc{C} \to \mathrm{Ab}$ given as the sheafification of
\[
X \mapsto \oo(X)^\times / q^{\Z}.
\] 
This basically boils down to showing that a morphism $X \to E_q$ can be lifted, locally on $X$, to a morphism to $\Gm$, unique up to translation by a power of $q$. But this is clear since $\Gm \to E_q$ is a covering map (given by factoring out translations of powers of $q$).
\end{proof}

So, we have showed that $E_q$ is a proper, one-dimensional, commutative group rigid space. In analogy with complex tori, we might expect that $E_q$ is the analytification of an elliptic curve. This is true, and can be proved in a few different ways (analogous to the complex case). We summarize its properties here. Recall the power series expansion
\[
j(q) = q^{-1} + 744 + 196884q + \dots \in \Z \lb q \rb
\]
of the $j$-invariant from the theory of elliptic curves and modular functions.

\begin{theo}
$E_q$ is an elliptic curve with $j$-invariant $j(E_q)=j(q)$ (viewed as a convergent power series in $K$); in particular $|j(E_q)| > 1$. Conversely, if $E/K$ is an elliptic curve with $j(E) \in K$ and $|j(E)|>1$, then there is a unique $q\in K$ with $0<|q|<1$ such that $E \cong E_q$ over a finite (in fact at most quadratic) extension of $K$.
\end{theo}

We refer to \cite[\S 5.1]{fresnel-vanderput} for a proof and further discussion. In preparation for a general discussion of uniformization of curves, let us describe how one can create formal models of $E_q$ with relatively explicit special fibres. First, let us include picture of $\mb{G}_m$ and $E_q$, in the spirit\footnote{We have choosen to incorporate one more type $5$ point in this and subsequent pictures.} of the picture at the end of Lecture 3:

\begin{figure}[htp]
\centering
\includegraphics{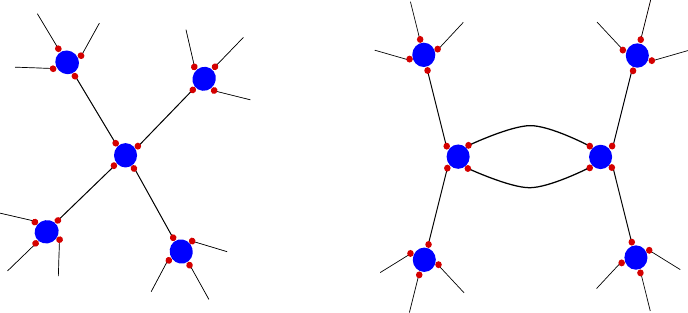}
\caption{$\Gm$ and $E_q$.}
\end{figure}

On the left we have $\mb{G}_m$: We think of the upper right type $2$ point as the Gauss point of the closed unit disc, i.e. the point corresponding to the disc $|T|\leq 1$. The lower left type $2$ corresponds to the disc $|T|\leq |q|$, and the middle type $2$ corresponds to the disc $|T|\leq |q|^{1/2}$. On the right we have $E_q$: The points corresponding to $|T|\leq |q|$ and $|T|\leq 1$ have been identified and correspond to the left middle type $2$ point in the picture of $E_q$; the right middle type $2$ corresponds to $|T|\leq |q|^{1/2}$ in $\mb{G}_m$. The upper left and lower right type $2$ points in $\mb{G}_m$ map to the right upper and right lower type $2$ points in $E_q$, respectively.

Now consider the open affinoids
\[
W_1^\prime = \{ |q| \leq |T| \leq |q|^{1/2} \}, \,\,\,\,  W_2^\prime = \{ |q|^{1/2} \leq |T| \leq 1 \}
\]
inside $\Gm$. The $W_i^\prime$ map isomorphically into $E_q$ and we set $W_i = \pi(W_i^\prime)$. The intersections between these sets is
\[
W_{12} := W_1 \cap W_2 \cong \{ |T| = |q| \} \sqcup \{ |T| = 1 \}.
\]
We illustrate these affinoids by recolouring the picture above: 

\begin{figure}[htp]
\centering
\includegraphics{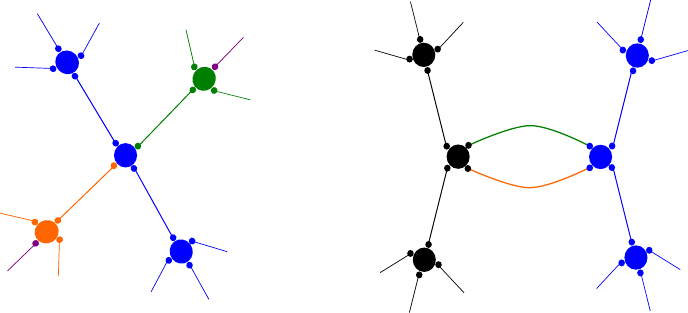}
\caption{$\Gm$ and $E_q$ with affinoids coloured.}
\end{figure}

Let us first describe $\mb{G}_m$ (on the left). The blue locus is the intersection $W_1^\prime \cap W_2^\prime$, the orange locus is $ W_1^\prime \setminus W_2^\prime$, the green locus is $W_2^\prime \setminus W_1^\prime$, and the purple locus is then $\mb{G}_m \setminus (W_1^\prime \cup W_2^\prime)$. For $E_q$ (on the right), the blue locus is the image of the blue locus from $\mb{G}_m$. The black locus is the locus that is hit by \emph{both} $W_1^\prime$ and $W_2^\prime$; the orange locus is the locus that is \emph{only} hit by $W_1^\prime$ and the green locus is the locus that is \emph{only} hit by $W_2^\prime$. 

To obtain a formal model from the covering $E_q = W_1 \cup W_2$, we start by defining affine formal schemes $\mf{W}_i = \Spf \oo(W_i)^\circ$ and $\mf{W}_{12} = \Spf \oo(W_{12})^\circ$. Before going further, let us describe these formal schemes a bit more.

\begin{lemm}\label{reductions of annuli}
Let $\lambda,\mu \in K^\times$ with $|\lambda|\leq |\mu|$ and let $X = \{ |\lambda| \leq |T| \leq |\mu|\} \sub \Gm$ be a closed annulus. Then $\oo(X) = K\langle \mu^{-1}T, \lambda T^{-1} \rangle$ and $\oo(X)^\circ = K^\circ \langle \mu^{-1}T, \lambda T^{-1} \rangle$. Let $k$ is the residue field of $K$. If $|\lambda| < |\mu|$, then the reduction of $\Spf \oo(X)^\circ$ is isomorphic to $\Spec k[s,t]/(st)$, and if $|\lambda|=|\mu|$ then the reduction of $\Spf \oo(X)^\circ$ is isomorphic to $\Spec k[s,s^{-1}]$.
\end{lemm}

\begin{proof}
We only sketch the proof. Rescaling, we may assume that $\mu=1$. Essentially by definition we have $\oo(X) = K \langle T,\lambda T^{-1} \rangle$, which we may rewrite as $K \langle s,t \rangle / (st -\lambda)$. Some calculation shows that $\oo(X)^\circ = K^\circ \langle s,t \rangle / (st -\lambda)$, and from this the statement about reductions follows.
\end{proof}

Since $E_q$ is obtained by gluing the affinoids $W_1$ and $W_2$ along the intersection $W_{12}$, we can now try to create a formal model of $E_q$ by gluing $\mf{W}_1$ and $\mf{W}_2$ along $\mf{W}_{12}$. For this to work, one has to check that the natural maps $\mf{W}_{ij} \to \mf{W}_i$ are open immersions, which we leave to the reader. This gives us a formal model $\mf{E}$ of $E_q$. Its special fibre is two projective lines intersecting in two ordinary double points (this follows from the descriptions from Lemma \ref{reductions of annuli} and the descriptions of the maps $\mf{W}_{ij} \to \mf{W}_i$). In terms of our previous picture of $E_q$, we can illustrate the specialization map from $E_q$ to $\mf{E}$ as follows:

\begin{figure}[htp]
\centering
\includegraphics{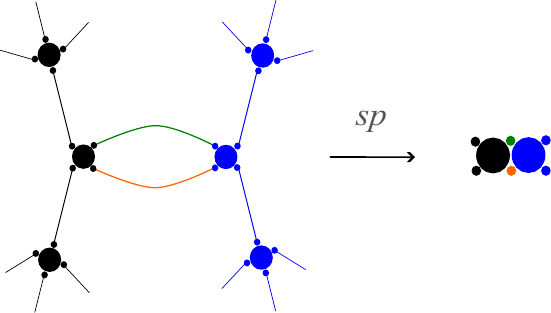}
\caption{The specialization map from $E_q$ to $\mf{E}$.}
\end{figure}

Note that in our pictures of $\mb{G}_m$ and $E_q$, each type $2$ point together with its associated type $5$ points make up the topological space of a projective line over the residue field $k$ of $K$, so the space on the right above is a (rather nonstandard) illustration of two projective lines intersecting in two points.

The formal scheme $\mf{E}$ is an example of a semistable model of a curve. We now recall this notion.

\begin{defi}
Let $X$ be a nonsingular projective curve over $K$. A semistable model of $X$ is a proper admissible formal scheme $\mf{X}$ over $K^\circ$ such that the special fibre $\ol{X}$ is reduced, any singularity is a nodal singularity\footnote{This is also called an ordinary double point singularity -- it means that the completed local ring is isomorphic to $k\lb s,t \rb / (st)$.}, and any irreducible component isomorphic to $\mathbb{P}^1$ contains at least two singular points of $\ol{X}$.
\end{defi}

Every nonsingular projective curve has a semistable model, according to the famous semistable reduction theorem.

\begin{theo}
After a finite extension of $K$, any nonsingular projective curve $X$ over $K$ admits a semistable model $\mf{X}$.
\end{theo} 

\begin{exer}
By adapting the construction above, convince yourself that $E_q$ has a semistable model with three irreducible components ``intersecting in a triangle''. By inspecting the structure of the adic space of $\Gm$, convince yourself that any irreducible component of any semistable model of $E_q$ must have genus $0$.
\end{exer}

For the remainder of this lecture we will discuss uniformization of arbitrary nonsingular projective curves. To do this, it is beneficial to use the (topological space of the) Berkovich space. We give a general topological definition of ``Berkovich spaces'', following \cite{henkel}.

\begin{defi} Let $X$ be a locally spectral space. 
\begin{enumerate}
\item We say that $X$ is a valuative space if, for any $x\in X$, the set of generalizations of $x$ is totally ordered under the specialization/generalization relation. In particular, this implies that any $x\in X$ has a unique maximal generalization (exercise below). 

\item Let $X$ be a valuative space and write $X^B \sub X$ for the subset of maximal points (i.e. those that don't have any proper generalizations). Since any $x\in X$ has a unique maximal generalization this determines a function $X \to X^B$ whose restriction to $X^B$ is the identity. We equip $X^B$ with the quotient topology with respect to $X \to X^B$.
\end{enumerate}
\end{defi}

\begin{exer}
Prove the claim that every point in a valuative space has a unique maximal generalization.
\end{exer}

In \cite{henkel}, $X^B$ is called the separated quotient of $X$ and is denoted by $[X]$. Any analytic adic space $X$ is a valuative space and if $X$ is a rigid space, then $X^B$ is the corresponding Berkovich space (hence the choice of notation). We refer to \cite[\S 5]{henkel} for topological aspects of $X^B$, and only note here that $X^B$ is Hausdorff and locally compact when $X$ is taut.

Let us now sketch the uniformization of general nonsingular projective curves over $K$. We strongly recommend the reader to consult Baker's wonderful article \cite{baker} for more explanations and intuition. For proofs and other perspectives, we refer to \cite{berkovich-book} and \cite{fresnel-vanderput}, respectively. Let $C$ be such a curve, with analytification $X:=C^{an}$ and Berkovich space $X^B$. To analyse $X$ and $X^B$, we will use a (fixed) semistable model $\mf{C}$ of $C$ over $K^\circ$ (enlarging $K$ if necessary) whose special fibre we denote by $\ol{C}$. Write
\[
\ol{C} = Z_1 \cup \dots \cup Z_r
\]
as the union of its irreducible components. The first observation that we need is the following.

\begin{theo} Consider the setup as above.
\begin{enumerate}
\item If $Y \to X$ is a covering map, then $Y^B \to X^B$ is a covering map.

\item If $r=1$, i.e. if $C$ has nonsingular reduction, then $X^B$ is contractible. In particular, there are no nontrivial topological covers of $X$.
\end{enumerate}
\end{theo} 

More generally, if $U=\Spa A$ is a smooth connected affinoid one-dimensional rigid space with smooth canonical reduction\footnote{The canonical reduction of an affinoid space $\Spa A$ is $\Spf A^\circ$, assuming that $A^\circ$ is topologically finitely generated over $K^\circ$.}, then there is a unique \emph{multiplicative} norm (up to equivalence) defining the topology on $A$. Since this norm is multiplicative, it defines a rank one point $\zeta\in U$. $\zeta$ is the unique point that maps onto the generic point of the canonical reduction, and $U^B$ contracts onto $\zeta$.

\begin{exer}
Visualize this when $U= \Spa K\langle T \rangle$. Which point is $\zeta$ in this case?
\end{exer}

Let us now return to the situation of a general nonsingular projective curve $C$ and its semistable reduction $\mf{C}$ with reduction $\ol{C} = Z_1 \cup \dots \cup Z_r$. Let $\zeta_i$ be the generic point of $Z_i$. Under the specialization map $sp : X \to \ol{C}$, each $\zeta_i$ lifts \emph{uniquely} to a point in $X$, which lies in $X^B$ and which we will denote by $\zeta_i$ as well. Each node $z\in \ol{C}$ is contained in precisely two irreducible components $Z_i$ and $Z_j$. The preimage $sp^{-1}(z)$ turns out to be isomorphic to an open annulus, and its Berkovich space contains a unique path connecting $\zeta_i$ and $\zeta_j$ which maps down to $z$ under the specialization map. By taking the union of the $\zeta_i$ and all the paths corresponding to an intersection point, we obtain a subspace $\Sigma \sub X^B$ which is a topological multigraph, isomorphic to the dual graph of $\ol{C}$\footnote{We recall that the dual graph of the semistable reduction has nodes corresponding to the irreducible components, and every intersection point between two components defines an edge between the corresponding nodes. It is a (undirected) multigraph.}. Moreover, the whole of $X^B$ deformation retracts onto $\Sigma$. In particular, we obtain the following:

\begin{theo}
With setup as above, let $T \to \Sigma$ be the universal covering space (which is a topological tree). Consider the composition $X \to X^B \to \Sigma$, where the second map is the retraction. Then the fibre product $Y = X \times_\Sigma T$ is a covering map of $X$. Its Berkovich space $Y^B$ is equal to  $X^B \times_\Sigma T$ and deformation retracts onto $T$, and is therefore simply connected.
\end{theo} 

The space $Y$ may reasonably be called the universal cover of $X$. For a way of constructing it without explicitly mentioning $X^B$, see \cite[\S 5.7]{fresnel-vanderput}.

To finish, let us illustrate the above procedure when $X$ is the Tate curve $E_q$ and we have selected the formal model $\mf{E}$ described above. As we have seen, the reduction $\ol{E}$ is a union of two components $Z_1$ and $Z_2$, both isomorphic to $\mb{P}^1$. Let $\zeta_i$ be the generic point of $Z_i$ for $i=1,2$. In terms of the picture above of the specialization map $E_q \to \mf{E}$, let us say that $\zeta_1$ corresponds to the black generic point of $\mf{E}$ and $\zeta_2$ corresponds to the blue generic point. We then see that their (unique) lifts to $E_q$ are the middle black and middle blue points, respectively. The dual graph of $\ol{E}$ is drawn below:

\begin{figure}[htp]
\centering
\includegraphics{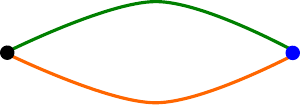}
\caption{The dual graph of $\ol{E}$.}
\end{figure}

The two nodes correspond to the two irreducible components, illustrated by the same colours as in the picture above, and the two edges correspond to the two intersection points. The Berkovich space $E_q^B$ is obtained from the topological space of the adic space $E_q$ by contracting the rank $2$ points to their associated rank $1$ points, illustrated below:

\begin{figure}[htp]
\centering
\includegraphics{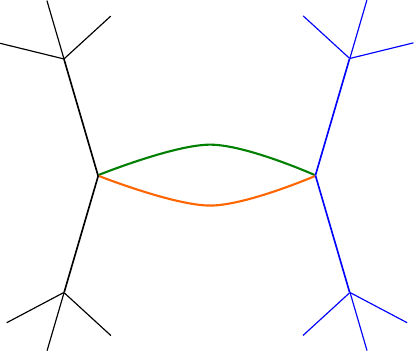}
\caption{The Berkovich space of $E_q$.}
\end{figure}

We see the dual graph here in the picture, as the loop in the middle. We also see that the whole of $E_q^B$ contracts onto the dual graph, as asserted. The universal cover of $E_q^B$ is the infinite ``tree'' $\mb{G}_m^B$, and one sees that the universal cover of $E_q$, as defined above, is $\mb{G}_m = E_q \times_{E_q^B} \mb{G}_m^B$, recovering the starting point of this lecture.

\newpage

\appendix
\counterwithin{theo}{section}

\section{\v{C}ech and sheaf cohomology}\label{appendix on coh}

Here we will give a review of the relationship between \v{C}ech and sheaf cohomology. While reasonably well known, we wish to illustrate that the relationship follows largely from general principles and the homological properties of \v{C}ech cohomology, which are perhaps less well known. For basic homological algebra, we will mostly reference \cite[\S 2]{weibel} and \cite[\S 3.1-3.2, 3.4]{hartshorne}, the latter being a very concise introduction to a lot of the basics needed for sheaf theory in algebraic geometry and similar subjects. Our treatment will use the Grothendieck spectral sequence heavily, so we state it here for completeness.

\begin{theo}[Grothendieck spectral sequence]\label{gss}
Let $\mc{A}$, $\mc{B}$ and $\mc{C}$ be abelian categories such that $\mc{A}$ and $\mc{B}$ have enough injectives. Let $F : \mc{A} \to \mc{B}$ and $G : \mc{B} \to \mc{C}$ be left exact functors and assume that $F$ sends injectives to $G$-acyclic objects. Then there is a (first quadrant, convergent) spectral sequence
\[
E_2^{ij} = R^i G (R^j F(A)) \implies R^{i+j}(G \circ F)(A)
\]
for any $A\in \mc{A}$. 
\end{theo} 

A reference is \cite[Theorem 5.8.3]{weibel}. We often use the following much easier special cases, which do not require the machinery of spectral sequences to state or prove.

\begin{coro}\label{easy gss}
Let $\mc{A}$, $\mc{B}$ and $\mc{C}$ be abelian categories such that $\mc{A}$ and $\mc{B}$ have enough injectives. Let $F : \mc{A} \to \mc{B}$ and $G : \mc{B} \to \mc{C}$ be left exact functors and let $A \in \mc{A}$.
\begin{enumerate}
\item If $G$ is exact, then $R^i(G\circ F)(A) \cong G(R^i F(A))$ for all $i$.

\item If $F$ is exact and sends injectives to $G$-acyclic objects, then $R^i(G\circ F)(A) \cong R^i G(F(A))$ for all $i$.
\end{enumerate}
\end{coro}

Let $X$ be a topological space; we let $\PSh(X)$ and $\Sh(X)$ denote the abelian categories of (abelian) presheaves and sheaves, respectively, on $X$. Recall that the inclusion $\iota : \Sh(X) \to \PSh(X)$ is fully faithful and left exact, and sheafification $\mc{F} \to \mc{F}^+$ is exact $\PSh(X) \to \Sh(X)$ and a left adjoint to $\iota$ (which implies that $\iota$ preserves injectives). The derived functors of $\iota$ have the following well known description.

\begin{exer}\label{derived functors of inclusion}
Let $\mc{F}$ be a sheaf on $X$. Show that $R^i \iota(\mc{F})$ is the presheaf $U \mapsto H^i(U,\mc{F})$, and that the sheafification of $R^i \iota(\mc{F})$ is zero for $i\geq 1$. Hint: Consider the compositions 
\[
\Sh(X) \overset{\iota}{\longrightarrow} \PSh(X) \overset{\mc{G} \mapsto \mc{G}(U)}{\longrightarrow} \Ab
\]
and
\[
\Sh(X) \overset{\iota}{\longrightarrow} \PSh(X) \overset{\mc{G} \mapsto \mc{G}^+}{\longrightarrow} \Sh(X).
\]
\end{exer}

\medskip

We now recall the basics of \v{C}ech cohomology. Let $\mf{U} = (U_i)_{i\in I}$ be an open cover of $X$. For distinct indices $i_1,\dots,i_n$ we will write $U_{i_1\dots i_n}$ for the intersection $U_{i_1} \cap \dots \cap U_{i_n}$. Recall that the \v{C}ech complex $\C(\mf{U},\mc{F})$ of a presheaf $\mc{F}$ with respect to $\mf{U}$ is defined as
\[
\prod_i \mc{F}(U_i) \to \prod_{i_1,i_2} \mc{F}(U_{i_1 i_2}) \to \prod_{i_1,i_2,i_3} \mc{F}(U_{i_1 i_2 i_3}) \to \dots \to \prod_{i_1,\dots,i_n} \mc{F}(U_{i_1 \dots i_n}) \to \dots .
\]
The cohomology groups $\cH^i(\mf{U},\mc{F}):= H^i(\C(\mf{U},\mc{F}))$ of the \v{C}ech complex are the \emph{\v{C}ech cohomology groups}. For a fixed open cover $\mf{U}$, the \v{C}ech complex $\mc{F} \mapsto \C(\mf{U},\mc{F})$ defines an \emph{exact} functor
\[
\PSh(X) \to \Ch^+(\Ab),
\]
where $\Ch^+(\Ab)$ denotes the abelian category of cochain complexes of abelian groups concentrated in non-negative degrees.

\begin{exer}
Prove that $\Ch^+(\Ab)$ has enough injectives which are acyclic in degrees $\geq 1$, as follows\footnote{Indeed, you can take it as a further exercise to prove that all injective objects in $\Ch^+(\Ab)$ are acyclic in degrees $\geq 1$.}: Given $M\in \Ab$, define complexes $K^n(M)$ for $n\geq 0$ by letting $K^0(M)$ be $M$ itself viewed as a complex in degree $0$, and, for $n\geq 1$, letting $K^n(M)$ be the complex
\[
\dots \to 0 \to M \overset{id_M}{\to} M \to 0 \to \dots
\]
with the first $M$ in degree $n-1$ and the second in degree $n$. If $C^\bu \in \Ch^+(\Ab)$, choose an injection $C^n \to J^n$ into an injective abelian group $J^n$ for each $n\geq 0$ and construct an injection $C^\bu \to I^\bu := \prod_n K^n(J^n)$. Moreover, show that $I^\bu$ is injective and that $H^i(I^\bu) =0$ for $i\geq 1$. 
\end{exer}

We now briefly discuss the functoriality of $\C(\mf{U},\mc{F})$ with respect to $\mf{U}$, refering to \cite[Tag 09UY]{stacks-project} for details. If $\mf{V} = (V_j)_{j\in J}$ is another cover of $X$, we say that $\mf{V}$ is a \emph{refinement} of $\mf{U}$ if there is a function $f : J \to I$ of index sets such that $V_j \sub U_{f(j)}$ for all $j\in J$. The set of open covers of $X$ form a directed set under refinements. A choice of $f$ as above induces a chain map $\phi_f : \C(\mf{U},\mc{F}) \to \C(\mf{V},\mc{F})$, and if $g : J \to I$ is another function satisfying $V_j \sub U_{g(j)}$ for $j\in J$, then $\phi_f$ and $\phi_g$ are chain homotopic. In particular, we get a well defined map $\phi : \C(\mf{U},\mc{F}) \to \C(\mf{V},\mc{F})$ in the homotopy category $\K^+(\Ab)$ of $\Ch^+(\Ab)$, and hence well defined maps $\cH^i(\mf{U},\mc{F}) \to \cH^i(\mf{V},\mc{F})$ on \v{C}ech cohomology groups. We then define
\[
\cH^i(X,\mc{F}) := \varinjlim_{\mf{U}} \cH^i(\mc{U},\mc{F}).
\]
The homological properties of \v{C}ech cohomology are summarized in the next proposition. For its proof, we recall that the framework for (right) derived functors in homological algebra pre-derived categories is that of (covariant, cohomological) $\delta$-functors \cite[\S 2.1]{weibel}. A $\delta$-functor between abelian categories $\mc{A}$ and $\mc{B}$ is, roughly speaking, a collection of additive functors $(T^n : \mc{A} \to \mc{B})_{n \in \Z_{\geq 0}}$ and satisfying certain axioms \cite[Definition 2.1.1]{weibel}. For a left exact functor $F$ out of an abelian category with enough injectives, the right derived functors $R^iF$ are characterized by $(R^iF)_i$ being a \emph{universal} $\delta$-functor \cite[Corollary III.1.4]{hartshorne}. Moreover, universality is implied by the notion of effaceability \cite[Theorem 1.3A]{hartshorne}.

\begin{prop}\label{cech is a derived functor}
Let $\mf{U}$ be an open cover of $X$. The functors $\cH^0(\mf{U},-)$ and $\cH^0(X,-)$ from $\PSh(X)$ to $\Ab$ are left exact, and their derived functors are given by $R^i \cH^0(\mf{U},-) = \cH^i(\mf{U},-)$ and $R^i \cH^0(X,-) = \cH^i(X,-)$.
\end{prop}

\begin{proof}
We first prove the statements about $\cH^i(\mf{U},-)$. We may write $\cH^0(\mf{U},-)$ as the composition
\[
\PSh(X) \overset{\C(\mf{U},-)}{\longrightarrow} \Ch^+(\Ab) \overset{H^0}{\longrightarrow} \Ab.
\]
The \v{C}ech complex functor is exact and $H^0 : \Ch^+(\Ab) \to \Ab$ is left exact, with derived functors $R^iH^0 = H^i$ \cite[Exercise 2.4.4]{weibel}, so by Corollary \ref{easy gss} it suffices to show that $\cH^i(\mf{U},\mc{I}) = 0$ for all $i \geq 1$ and all injective presheaves $\mc{I}$. For this, we refer to the proof of \cite[Tag 01EN]{stacks-project}.

\medskip

It remains to prove the statements about $\cH^i(X,-)$, for which it suffices to prove that $(\cH^i(X,-))_i$ is a universal $\delta$-functor. Since the $(\cH^i(\mf{U},-))_i$ are $\delta$-functors, one easily checks that $(\cH^i(X,-))_i$ is a $\delta$-functor. Moreover, if $\mc{F}$ is any presheaf on $X$, then we can choose an injection $\mc{F} \to \mc{I}$ into an injective presheaf, and by above we see that $\cH^i(X,\mc{I}) = \varinjlim \cH^i(\mf{U},\mc{I}) = 0$ for $i\geq 1$. Therefore $(\cH^i(X,-))_i$ is effaceable, and hence universal as desired.
\end{proof}

As a consequence, we obtain the \v{C}ech-to-derived functor spectral sequence(s).

\begin{theo}\label{cech to derived}
For each open cover $\mf{U}$ and each sheaf $\mc{F}$ on $X$ we have a spectral sequence
\[
E_2^{ij} = \cH^i(\mf{U},R^j\iota(\mc{F})) \implies H^{i+j}(X,\mc{F}).
\]
Moreover, for each sheaf $\mc{F}$ we have a spectral sequence
\[
E_2^{ij} = \cH^i(X,R^j\iota(\mc{F})) \implies H^{i+j}(X,\mc{F}).
\]
\end{theo}

\begin{proof}
Note that $\cH^0(\mc{U},\mc{F}) = \cH^0(X,\mc{F}) = H^0(X,\mc{F})$ for any \emph{sheaf} $\mc{F}$, so by Proposition \ref{cech is a derived functor} the spectral sequences in the theorem are the Grothendieck spectral sequences for the compositions
\[
\Sh(X) \overset{\iota}{\longrightarrow} \PSh(X) \overset{\cH^0(\mf{U},-)}{\longrightarrow} \Ab
\]
and
\[
\Sh(X) \overset{\iota}{\longrightarrow} \PSh(X) \overset{\cH^0(X,-)}{\longrightarrow} \Ab
\]
(recall that $\iota$ preserves injectives).
\end{proof}

Let us now list the key consequences that we will need. We start with working out what happens in some low degrees.

\begin{coro}\label{low degrees}
Let $\mc{F}$ be a sheaf on $X$. Then we have $\cH^0(X,R^j\iota(\mc{F}))=0$ for $j\geq 1$. As a consequence, the natural map $\cH^1(X,\mc{F}) \to H^1(X,\mc{F})$ is an isomorphism.
\end{coro}

\begin{proof}
That $\cH^0(X,R^j\iota(\mc{F}))=0$ for $j\geq 1$ follows from the fact that the sheafification of $R^j\iota(\mc{F})$ is zero for $j\geq 1$ (Exercise \ref{derived functors of inclusion}). That $\cH^1(X,\mc{F}) \to H^1(X,\mc{F})$ is an isomorphism then follows from the \v{C}ech-to-derived functor spectral sequence (look at its five term exact sequence\footnote{The five term exact sequence is stated in \cite[Theorem 5.8.3]{weibel}}).
\end{proof}

Next, we come to Leray's Theorem, which gives a criterion for when \v{C}ech cohomology for a specific cover computes sheaf cohomology.

\begin{coro}[Leray's Theorem]\label{Leray}
Let $\mf{U}=(U_i)_{i\in I}$ be an open cover of $X$ and let $\mc{F}$ be a sheaf on $X$. Assume that $H^j(U_{i_1\dots i_n},\mc{F})=0$ for all distinct $i_1,\dots,i_n \in I$ and $j\geq 1$. Then the natural map $\cH^j(\mf{U},\mc{F}) \to H^j(X,\mc{F})$ is an isomorphism for all $j\geq 0$.
\end{coro}

\begin{proof}
By Theorem \ref{cech to derived} it will suffice to show that $\cH^j(\mf{U},R^k\iota(\mc{F})) = 0$ whenever $k \geq 1$. But the Cech complex $\C^\bu(\mf{U},R^k\iota(\mc{F}))$ is
\[
\prod_i H^k(U_i,\mc{F}) \to \prod_{i_1,i_2} H^k(U_{i_1 i_2},\mc{F}) \to \dots \to \prod_{i_1,\dots,i_n} H^k(U_{i_1 \dots i_n},\mc{F}) \to \dots
\]
by the description of $R^k\iota(\mc{F})$ from Exercise \ref{derived functors of inclusion}, so all terms of $\C^\bu(\mf{U},R^k\iota(\mc{F}))$ vanish when $k\geq 1$ by assumption. This finishes the proof.
\end{proof}

The disadvantage of Leray's Theorem is that the condition is phrased in terms of sheaf cohomology. Since the development of cohomology of adic spaces starts with us only being able to compute \v{C}ech cohomology, we instead need the following result of Cartan to get us off the ground. Note that to define the cover-independent \v{C}ech cohomology groups $\cH^i(X,-)$ we do not need all open covers $\mf{U}$, any cofinal set of open covers will do (since any direct limit can be computed using only a cofinal subset of the index set). In particular, if $\mf{B}$ is a basis of open sets of $X$ which is closed under intersection, then it suffices to take covers consisting of elements in $\mf{B}$.

\begin{coro}[Cartan]\label{Cartan}
Let $\mc{F}$ be a sheaf on $X$, and $\mf{B}$ be a basis of open sets of $X$ which is closed under intersection. Assume that $\cH^i(U,\mc{F}) =0$ for any open $U \in \mf{B}$ and any $i\geq 1$. Then the natural map $\cH^i(\mf{U},\mc{F}) \to H^i(X,\mc{F})$ is an isomorphism for any open cover $\mf{U} \sub \mf{B}$ and any $i \geq 0$. 
\end{coro}

\begin{proof}
Let $\mf{U}$ be an open cover with $\mf{U} \sub \mf{B}$. We will prove that $\cH^i(\mf{U},\mc{F}) \to H^i(X,\mc{F})$ is an isomorphism by checking that the criterion in Leray's Theorem holds. That means that we need to check $H^j(U, \mc{F})=0$ for any finite intersection $U$ of elements in $\mf{U}$ and any $j\geq 1$. In other words, Cartan's Theorem will follow if we can prove that $H^j(U, \mc{F})=0$ for any $U\in \mf{B}$ and any $j\geq 1$.

\medskip

We will prove this by induction on $j$. Recall that our assumption is that $\cH^j(U,\mc{F}) =0$ for $j\geq 1$ and any $U\in \mf{B}$. By Corollary \ref{low degrees} we therefore know that $H^1(U, \mc{F})=0$ for any $U \in \mf{B}$, so assume that $j\geq 2$ and that $H^k(U,\mc{F}) = 0$ for $1 \leq k < j$ and any $U\in \mf{B}$. Now pick an arbitrary $W \in \mf{B}$; we want to show that $H^j(W, \mc{F})=0$. Consider the \v{C}ech-to-derived functor spectral sequence
\[
E_2^{pq} = \cH^p(W,R^q\iota(\mc{F})) \implies H^{p+q}(W,\mc{F}),
\]
and consider the terms on the $E_2$-page with $p+q = j$. If $1 \leq q \leq j-1$, and $\mf{W} \sub \mf{B}$ is an open cover of $W$, then $\C^\bu(\mf{W},R^q\iota(\mc{F}))=0$ by the induction hyothesis and hence $\cH^p(W,R^q\iota(\mc{F})) = 0$. If $q=j$, then $p=0$ and hence $\cH^p(W,R^q\iota(\mc{F})) = 0$ by Corollary \ref{low degrees}. Finally, if $q=0$ then $p=j$ and $E_2^{j,0} = \cH^j(W,\mc{F}) = 0$ by assumption. Thus $E_2^{pq} = 0$ whenever $p+q = j$ and hence $H^j(W,\mc{F})=0$ as desired, finishing the induction.
\end{proof}

\section{Some topology}

Here we record a few scattered facts from the the theory of (locally) spectral spaces that we will need. We refer to \cite[Tag 08YF]{stacks-project} for basic definitions for spectral spaces.

\begin{defi}
Let $f : X \to Y$ be a map of locally spectral spaces. We say that $f$ is quasicompact if, for every quasicompact open $V \sub Y$, $f^{-1}(V)$ is quasicompact.
\end{defi}

Note that quasicompactness is local on the target. Moreover, recall that quasicompact maps between spectral spaces are called spectral. Not every map between spectral spaces is spectral, as the example $\Spa(\Zp \lb X \rb, \Zp \lb X \rb ) \to \Spa(\Zp,\Zp)$ shows (the preimage of the quasicompact open subset $\Spa(\Qp,\Zp) \sub \Spa(\Zp,\Zp)$ is not quasicompact). However, if we restrict to \emph{analytic} adic spaces, the situation is better.

\begin{exer} Let $f : X \to Y$ be a morphism of analytic adic spaces. Show the following:
\begin{enumerate}
\item If $X$ and $Y$ are affinoid with $\oo_X(X)$ and $\oo_Y(Y)$ Tate and $V \sub Y$ is a rational subset, then $f^{-1}(V)$ is rational subset. In particular, $f$ is quasicompact.

\smallskip

\item Assume that there exists a cover $(V_i)_{i\in I}$ of $Y$ by affinoid Tate opens $V_i$ such that $f^{-1}(V_i)$ is quasicompact for all $i$. Then $f$ is quasicompact.
\end{enumerate}
\end{exer}

We will use the following observation.

\begin{lemm}\label{quasicompactness of maps}
Let $f : X \to Y$ be a morphism of analytic adic spaces. If $X$ is quasicompact and $Y$ is quasiseparated, then $f$ is quasicompact.
\end{lemm}

\begin{proof}
Choose an open cover $(V_j)_{j\in J}$ of $Y$ by Tate affinoids. Since $X$ is quasicompact, we can choose a finite cover $(U_i)_{i\in I}$ of $X$ by Tate affinoids such that for all $i$ there is a $j$ such that $f(U_i) \sub V_j$. Now consider a quasicompact open subset $W \sub Y$. We have
\[
f^{-1}(W) = \bigcup_{i\in I} U_i \cap f^{-1}(W),
\]
so it suffices to show that each $U_i \cap f^{-1}(W)$ is quasicompact. Fix $i$ and choose $j$ such that $f(U_i) \sub V_j$. Then we see that $U_i \cap f^{-1}(W) = (f|_{U_i})^{-1}(W\cap V_j)$. Since $Y$ is quasiseparated, $V_j \cap W$ is quasicompact, and $f|_{U_i} : U_i \to V_j$ is quasicompact since $U_i$ and $V_j$ are Tate affinoids, so we conclude that $U_i \cap f^{-1}(W)$ is quasicompact as desired, which finishes the proof.
\end{proof}

To conclude this short section, we now give the proof of Lemma \ref{some topology}. We recall the statement first.

\begin{lemm}\label{some topology proof}
Let $X$ be a taut locally spectral space.
\begin{enumerate}
\item Let $S \sub X$ be a quasicompact subset. Then $\ol{S}$ is quasicompact. In particular, if $x\in X$, then $\ol{\{x\}}$ is quasicompact.

\smallskip

\item Let $S\sub X$ be a quasicompact subset. Then $x\in \ol{S}$ if and only if $x$ is a specialization of a point in $S$.

\smallskip

\item Assume that $X$ is the underlying topological space of an analytic adic space. Let $x \in X$ be a rank one point and let $\Sigma$ be the set of quasicompact open subsets containing $\ol{\{x\}}$. Then $\ol{\{x\}} = \bigcap_{U\in \Sigma} U$.
\smallskip

\item Assume that $X$ is the underlying topological space of an analytic adic space. Let $x \in X$ be a rank one point. Then $\ol{\{x\}} = \bigcap_{U\in \Sigma} \ol{U}$, where $\Sigma$ is as in (3).
\end{enumerate}
\end{lemm}

\begin{proof}
For (1), choose a quasicompact open subset $U\sub X$ containing $S$. Then $\ol{U}$ is quasicompact since $X$ is taut, so the closed subset $\ol{S} \sub \ol{U}$ is quasicompact as well. 

\medskip

For part (2), note that any $x$ which is a specialization of a point in $S$ has to be contained in $\ol{S}$, so we need to prove the converse. Assume that $x$ is not a specialization of a point in $S$. Let $\Delta$ be the set of all quasicompact open subsets containing $x$. Set $G= \bigcap_{V\in \Delta} V$; $G$ is the set of all generalizations of $x$ and hence $G \cap S = \emptyset$. Since $X$ is taut we may choose a quasicompact open subset $W$ containing $\ol{S} \cup G$; since $X$ is quasiseparated $W$ is spectral. Then we have $\bigcap_{V\in \Delta} (V \cap U) = G \cap U = \emptyset$ and the $V\cap U$ are closed in the constructible topology on $W$, which is compact and Hausdorff, so there must be a $V \in \Sigma$ such that $V \cap U = \emptyset$. It follows that $x\notin \ol{U}$, as desired.

\medskip

For part (3), note that if $x$ is a rank one point then $\ol{\{x\}}$ is closed under generalizations. By (1) it is also quasicompact, so there is a quasicompact open subset $W^\prime \sub X$ with $\ol{\{x\}} \sub V$, and $W^\prime$ is spectral since $X$ is quasiseparated. By \cite[Tag 0A31]{stacks-project} it follows that $\ol{\{x\}}$ is an intersection of quasicompact open subsets.

\medskip

For part (4), we need to show that any $z\in \bigcap_{U \in \Sigma}\ol{U}$ is contained in $\ol{\{x\}}$. Let $y$ be the unique rank one generalization of $z$ and let $U \in \Sigma$. Since $z \in \ol{U}$, we must have $y \in U$ by part (2). It follows that $y \in \bigcap_{U\in \Sigma} U$, which is equal to $\ol{\{x\}}$ by part (3), and hence $z\in \ol{\{x\}}$ as desired (and $y=x$).
\end{proof}

\section{Some algebra}\label{sec: appalgebra}

Here we recall a few flatness results that are useful for Exercises \ref{ex: flatness1} and \ref{ex: flatness2}. More broadly, they illustrate the usefulness of \emph{completed} local rings in rigid geometry. One place where this is discussed further (and used) is \cite{conrad-irreducible}. We begin by noting that flatness can be checked on completions.

\begin{theo}
Let $A$ be a Noetherian local ring and let $\wh{A}$ be its completion. Then the natural map $A \to \wh{A}$ is faithfully flat. Moreover, if $B$ is another Noetherian local ring and $A \to B$ is a local homomorphism, then $A \to B$ is flat if and only if the induced map $\wh{A} \to \wh{B}$ on completions is flat.
\end{theo}

\begin{proof}
Recall that, for a local homomorphism of local rings, flatness is equivalent to faithful flatness. Flatness of $A \to \wh{A}$ is then \cite[Theorem 8.8]{matsumura}. That flatness of $\wh{A} \to \wh{B}$ implies flatness of $A \to B$ then follows easily, and the converse follows from \cite[Theorem 22.4(i)]{matsumura}.
\end{proof}

The other fact we wish to recall is about checking flatness for ring maps that arise as direct limits. We first need the following lemma.

\begin{lemm}\label{modules are direct limit of fp modules}
Let $A$ be any ring and let $M$ be any $A$-module. Then $M$ can be written as a direct limit of finitely presented $A$-modules.
\end{lemm}

\begin{proof}
We sketch a proof. First, $M$ is the union of its finitely generated submodules (and this union is direct). Second, writing a finitely generated $A$-module $N$ as $A^n/V$ with $V\sub A^n$ generated by $\{v_i \mid i\in I\}$ say, $N$ is the direct limit $\varinjlim_J A^n/V_J$ of the finitely presented modules $A^n/V_J$, where $J$ ranges over the finite subsets of $I$ and $V_J\sub A^n$ is the submodule generated by $\{v_i \mid i\in J\}$. Combining these gives the result.
\end{proof}

\begin{prop}
Let $A = \varinjlim_{i\in I} A_i$ and $B = \varinjlim_{j \in J}B_j$ be two rings, written as direct limits of rings with flat transition maps. Assume that we have an order preserving map $f : I \to J$ and ring maps $\varphi_i : A_i \to B_{f(i)}$ for $i$, such that whenever $i < k$, the diagram
\[
\xymatrix{
A_i \ar[r]^{\varphi_i}\ar[d] & B_{f(i)}  \ar[d] \\
A_k \ar[r]^{\varphi_k} & B_{f(k)}
}
\]
commutes. Then we get an induced map $\varphi : A \to B$. The following holds:
\begin{enumerate}
\item The maps $A_i \to A$ are flat (and similarly for $B$).

\item If $\varphi_i$ is flat for all $i$, then $\varphi$ is flat as well.
\end{enumerate}
\end{prop}

\begin{proof}
A direct limit of flat modules is flat, so part (1) follows by writing $A = \varinjlim_{k\geq i} A_k$ and noting that all $A_k$ are flat $A_i$-modules by assumption. 

For part (2), it suffices to prove that $\Tor_q^A(M,B)=0$ for all $q\geq 1$ and all $A$-modules $M$. Since $\Tor$ commutes with direct limits\footnote{To see this, resolve in the other variable and use that direct limits are exact and commute with tensor products.} in either variable, it suffices to check this for finitely presented $M$ by Lemma \ref{modules are direct limit of fp modules}. For such $M$, we may find an $i\in I$ and a finitely presented $A_i$-module $M_i$ such that $M = M_i \otimes_{A_i}A$ (write $M$ as the cokernel of a finite matrix). Using flatness of $A_i \to A$, it then follows that
\[
\Tor_q^A(M,B) = \Tor_q^{A_i}(M_i,B).
\]
Finally, $A_i \to B$ is flat since it is the composition $A_i \to B_{f(i)} \to B$ of flat maps, so these $\Tor$-groups vanish when $q\geq 1$, as desired.
\end{proof}

\newpage



\bibliographystyle{alpha}
\bibliography{heidelbergbib}

\begin{thebibliography}{BCKW19}

\bibitem[AIP15]{aip-siegel}
Fabrizio Andreatta, Adrian Iovita, and Vincent Pilloni.
\newblock {$p$}-adic families of {S}iegel modular cuspforms.
\newblock {\em Ann. of Math. (2)}, 181(2):623--697, 2015.

\bibitem[AIP18]{aip-halo}
Fabrizio Andreatta, Adrian Iovita, and Vincent Pilloni.
\newblock Le halo spectral.
\newblock {\em Ann. Sci. \'{E}c. Norm. Sup\'{e}r. (4)}, 51(3):603--655, 2018.

\bibitem[ALY23]{achinger-lara-youcis}
Piotr Achinger, Marcin Lara, and Alex Youcis.
\newblock Geometric arcs and fundamental groups of rigid spaces.
\newblock {\em J. Reine Angew. Math.}, 799:57--107, 2023.

\bibitem[And21]{andreychev}
Grigory Andreychev.
\newblock Pseudocoherent and {P}erfect {C}omplexes and {V}ector {B}undles on
  {A}nalytic {A}dic {S}paces.
\newblock \url{https://arxiv.org/abs/2105.12591}, 2021.

\bibitem[Bak08]{baker}
Matthew Baker.
\newblock An introduction to {B}erkovich analytic spaces and non-{A}rchimedean
  potential theory on curves.
\newblock In {\em {$p$}-adic geometry}, volume~45 of {\em Univ. Lecture Ser.},
  pages 123--174. Amer. Math. Soc., Providence, RI, 2008.

\bibitem[BCKW19]{aws-perfectoid}
Bhargav Bhatt, Ana Caraiani, Kiran~S. Kedlaya, and Jared Weinstein.
\newblock {\em Perfectoid spaces}, volume 242 of {\em Mathematical Surveys and
  Monographs}.
\newblock American Mathematical Society, Providence, RI, 2019.
\newblock Lectures from the 2017 Arizona Winter School, held in Tucson, AZ,
  March 11--17, Edited and with a preface by Bryden Cais, With an introduction
  by Peter Scholze.

\bibitem[Ber90]{berkovich-book}
Vladimir~G. Berkovich.
\newblock {\em Spectral theory and analytic geometry over non-{A}rchimedean
  fields}, volume~33 of {\em Mathematical Surveys and Monographs}.
\newblock American Mathematical Society, Providence, RI, 1990.

\bibitem[Ber24]{bergdall-heid}
John Bergdall.
\newblock Huber rings and valuation spectra.
\newblock {\em This volume}, 2024.

\bibitem[Bos14]{bosch}
Siegfried Bosch.
\newblock {\em Lectures on formal and rigid geometry}, volume 2105 of {\em
  Lecture Notes in Mathematics}.
\newblock Springer, Cham, 2014.

\bibitem[Buz07]{buzzard-eigenvarieties}
Kevin Buzzard.
\newblock Eigenvarieties.
\newblock In {\em {$L$}-functions and {G}alois representations}, volume 320 of
  {\em London Math. Soc. Lecture Note Ser.}, pages 59--120. Cambridge Univ.
  Press, Cambridge, 2007.

\bibitem[Chi90]{chiarellotto}
Bruno Chiarellotto.
\newblock Duality in rigid analysis.
\newblock In {\em {$p$}-adic analysis ({T}rento, 1989)}, volume 1454 of {\em
  Lecture Notes in Math.}, pages 142--172. Springer, Berlin, 1990.

\bibitem[Con99]{conrad-irreducible}
Brian Conrad.
\newblock Irreducible components of rigid spaces.
\newblock {\em Ann. Inst. Fourier (Grenoble)}, 49(2):473--541, 1999.

\bibitem[Con06]{conrad-relative-ampleness}
Brian Conrad.
\newblock Relative ampleness in rigid geometry.
\newblock {\em Ann. Inst. Fourier (Grenoble)}, 56(4):1049--1126, 2006.

\bibitem[CS24]{clause-scholze-lectures}
Dustin Clausen and Peter Scholze.
\newblock Analytic {S}tacks.
\newblock
  \url{https://www.youtube.com/playlist?list=PLx5f8IelFRgGmu6gmL-Kf_Rl_6Mm7juZO},
  2024.

\bibitem[dJ95]{dejong}
A.~J. de~Jong.
\newblock Crystalline {D}ieudonn\'e{} module theory via formal and rigid
  geometry.
\newblock {\em Inst. Hautes \'Etudes Sci. Publ. Math.}, (82):5--96, 1995.

\bibitem[Eme17]{emerton-locallyanalytic}
Matthew Emerton.
\newblock Locally analytic vectors in representations of locally {$p$}-adic
  analytic groups.
\newblock {\em Mem. Amer. Math. Soc.}, 248(1175):iv+158, 2017.

\bibitem[FvdP04]{fresnel-vanderput}
Jean Fresnel and Marius van~der Put.
\newblock {\em Rigid analytic geometry and its applications}, volume 218 of
  {\em Progress in Mathematics}.
\newblock Birkh\"{a}user Boston, Inc., Boston, MA, 2004.

\bibitem[Har77]{hartshorne}
Robin Hartshorne.
\newblock {\em Algebraic geometry}.
\newblock Graduate Texts in Mathematics, No. 52. Springer-Verlag, New
  York-Heidelberg, 1977.

\bibitem[Hen16]{henkel}
Timo Henkel.
\newblock A comparison of adic spaces and {B}erkovich spaces, 2016.

\bibitem[Hub94]{hu-generalization}
R.~Huber.
\newblock A generalization of formal schemes and rigid analytic varieties.
\newblock {\em Math. Z.}, 217(4):513--551, 1994.

\bibitem[Hub96]{hu-etale}
Roland Huber.
\newblock {\em \'{E}tale cohomology of rigid analytic varieties and adic
  spaces}.
\newblock Aspects of Mathematics, E30. Friedr. Vieweg \& Sohn, Braunschweig,
  1996.

\bibitem[H{\"u}b24]{hubner-heid}
Katharina H{\"u}bner.
\newblock Adic spaces.
\newblock {\em This volume}, 2024.

\bibitem[K\"74]{kopf}
Ursula K\"{o}pf.
\newblock \"{U}ber eigentliche {F}amilien algebraischer {V}ariet\"{a}ten
  \"{u}ber affinoiden {R}\"{a}umen.
\newblock {\em Schr. Math. Inst. Univ. M\"{u}nster (2)}, (Heft 7):iv+72, 1974.

\bibitem[Ked19]{kedlaya-aws}
Kiran~S. Kedlaya.
\newblock Sheaves, stacks, and shtukas.
\newblock In {\em Perfectoid spaces}, volume 242 of {\em Mathematical Surveys
  and Monographs}, pages 45--191. Amer. Math. Soc., Providence, RI, 2019.

\bibitem[KL15]{kedlaya-liu-i}
Kiran~S. Kedlaya and Ruochuan Liu.
\newblock Relative {$p$}-adic {H}odge theory: foundations.
\newblock {\em Ast\'{e}risque}, (371):239, 2015.

\bibitem[KL16]{kedlaya-liu-ii}
Kiran~S. Kedlaya and Ruochuan Liu.
\newblock Relative p-adic {H}odge theory, ii: {I}mperfect period rings, 2016.

\bibitem[Liu90]{liu}
Qing Liu.
\newblock Sur les espaces de {S}tein quasi-compacts en g\'{e}om\'{e}trie
  rigide.
\newblock {\em Tohoku Math. J. (2)}, 42(3):289--306, 1990.

\bibitem[Lud23]{ludwig-notes}
Judith Ludwig.
\newblock Topics in the theory of adic spaces.
\newblock \url{https://www.judith-ludwig.com/teaching}, 2023.

\bibitem[Lud24]{ludwig-heid}
Judith Ludwig.
\newblock Spectral theory and the {E}igenvariety machine.
\newblock {\em This volume}, 2024.

\bibitem[Mat89]{matsumura}
Hideyuki Matsumura.
\newblock {\em Commutative ring theory}, volume~8 of {\em Cambridge Studies in
  Advanced Mathematics}.
\newblock Cambridge University Press, Cambridge, second edition, 1989.
\newblock Translated from the Japanese by M. Reid.

\bibitem[Mik24]{mikami}
Yutaro Mikami.
\newblock Faithfully flat descent of quasi-coherent complexes on rigid analytic
  varieties via condensed mathematics.
\newblock {\em Int. Math. Res. Not. IMRN}, (8):7099--7128, 2024.

\bibitem[MP21]{maculan-poineau-stein}
Marco Maculan and J\'{e}r\^{o}me Poineau.
\newblock Notions of {S}tein spaces in non-{A}rchimedean geometry.
\newblock {\em J. Algebraic Geom.}, 30(2):287--330, 2021.

\bibitem[Poi10]{poineau-gaga}
J\'{e}r\^{o}me Poineau.
\newblock Raccord sur les espaces de {B}erkovich.
\newblock {\em Algebra Number Theory}, 4(3):297--334, 2010.

\bibitem[Sch12]{scholze-perfectoid}
Peter Scholze.
\newblock Perfectoid spaces.
\newblock {\em Publ. Math. Inst. Hautes \'{E}tudes Sci.}, 116:245--313, 2012.

\bibitem[Sch19]{scholze-condensed}
Peter Scholze.
\newblock Lectures on {C}ondensed {M}athematics.
\newblock \url{https://people.mpim-bonn.mpg.de/scholze/Condensed.pdf}, 2019.

\bibitem[Sch20]{scholze-analytic}
Peter Scholze.
\newblock Lectures on {A}nalytic {G}eometry.
\newblock \url{https://people.mpim-bonn.mpg.de/scholze/Analytic.pdf}, 2020.

\bibitem[Sch21]{scholze-mathoverflow}
Peter Scholze.
\newblock Answer to ``{C}ondensed criterion for sheafiness of adic spaces'',
  {M}ath{O}verflow.
\newblock
  \url{https://mathoverflow.net/questions/380235/condensed-criterion-for-sheafiness-of-adic-spaces},
  2021.

\bibitem[ST03]{schneider-teitelbaum-dist}
Peter Schneider and Jeremy Teitelbaum.
\newblock Algebras of {$p$}-adic distributions and admissible representations.
\newblock {\em Invent. Math.}, 153(1):145--196, 2003.

\bibitem[{Sta}18]{stacks-project}
The {Stacks Project Authors}.
\newblock \textit{Stacks Project}.
\newblock \url{https://stacks.math.columbia.edu}, 2018.

\bibitem[SW20]{berkeleylectures}
Peter Scholze and Jared Weinstein.
\newblock {\em Berkeley lectures on {$p$}-adic geometry}, volume 207 of {\em
  Annals of Mathematics Studies}.
\newblock Princeton University Press, Princeton, NJ, 2020.

\bibitem[Tem00]{temkin}
M.~Temkin.
\newblock On local properties of non-{A}rchimedean analytic spaces.
\newblock {\em Math. Ann.}, 318(3):585--607, 2000.

\bibitem[Wei94]{weibel}
Charles~A. Weibel.
\newblock {\em An introduction to homological algebra}, volume~38 of {\em
  Cambridge Studies in Advanced Mathematics}.
\newblock Cambridge University Press, Cambridge, 1994.

\bibitem[Zav23]{zavyalov-sheafiness}
Bogdan Zavyalov.
\newblock Sheafiness of strongy rigid-noetherian {H}uber pairs.
\newblock {\em Math. Ann.}, 386(3-4):1305--1324, 2023.

\bibitem[Zav24]{zavyalov-quotients}
Bogdan Zavyalov.
\newblock Quotients of admissible formal schemes and adic spaces by finite
  groups.
\newblock {\em Algebra Number Theory}, 18(3):409--475, 2024.

\end{thebibliography}









\end{document}